\documentclass[preprint,3p,12pt]{elsarticle}

\usepackage{amsmath,amssymb,amsfonts,amsthm}
\usepackage{graphicx}
\usepackage{subfigure}
\usepackage{float}

\usepackage{booktabs}
\usepackage{colortbl}
\usepackage{caption}
\usepackage{epstopdf}
\usepackage{algorithm}
\usepackage{algpseudocode}
\usepackage{array}
\usepackage{longtable}

\usepackage{bm}
\usepackage{makecell}
\usepackage{multirow}
\usepackage{multicol}

\newtheorem{theorem}{Theorem}[section]
\newtheorem{lemma}{Lemma}[section]

\newtheorem{proposition}[theorem]{Proposition}

\usepackage{lineno}
\modulolinenumbers[1]

\newtheoremstyle{noparens}%
  {}{}%
  {\itshape}{}%
  {\bfseries}{.}%
  { }%
  {\thmname{#1}\thmnumber{ #2}\mdseries\thmnote{ #3}}

\theoremstyle{noparens}
\newtheorem{lemmaNoParens}[lemma]{Lemma}

\def\bn{\mathbf{n}}

\def\fvec#1{\boldsymbol{#1}} 
\def\bvec#1{\mathbf{#1}} 
\def\bmat#1{\mathbf{#1}} 

\def\NIT{\texttt{NIT}}

\def\TIME{\texttt{TIME}}
\def\NP{\texttt{NP}}

\def\Tsetup{$\text{T}_{\text{setup}}$}
\def\Tsolve{$\text{T}_{\text{solve}}$}
\def\Ttotal{$\text{T}_{\text{total}}$}

\begin{document}


\begin{frontmatter}




\title{Two-Level preconditioning method for solving saddle point systems in contact computation}




\author[a,b]{Xiaoyu Duan}
\ead{duanxiaoyu19@gscaep.ac.cn}
\author[a,c]{Hengbin An\corref{cor1}}
\ead{an\_hengbin@iapcm.ac.cn}

\affiliation[a]{organization={Institute of Applied Physics and Computational Mathematics},
            city={Beijing},
            postcode={100094}, 
            country={China}}
            
\affiliation[b]{organization={Graduate School of China Academy of Engineering Physics},
            city={Beijing},
            postcode={100088}, 
            country={China}}

\affiliation[c]{organization={CAEP Software Center for High Performance Numerical Simulation},
            city={Beijing},
            postcode={100088}, 
            country={China}}
            

\begin{abstract}
In contact mechanics computation, the constraint conditions on the contact surfaces are typically enforced by the Lagrange multiplier method, resulting in a saddle point system. Given that the saddle point matrix is indefinite, 
solving these systems presents significant challenges. 
For a two-dimensional tied contact problem, an efficient two-level preconditioning method is developed. 
This method utilizes physical quantities for coarsening, introducing two 
types of interpolation operators and corresponding smoothing algorithms. 
Additionally, the constructed coarse grid operator exhibits symmetry and 
positive definiteness, adequately reflecting the contact constraints. 
Numerical results show the effectiveness of the method.
\end{abstract}




\end{frontmatter}



\section{Introduction}
\label{sect:intrd}

Contact problems in structural mechanics are commonly encountered in various
engineering fields, such as automotive design and manufacturing, as well as aerospace 
engine design and optimization~\cite{an2022shear, INTERNODES2022}. The finite element 
method (FEM) has been widely employed in the numerical simulation of contact 
problems~\cite{NTShughes1976finite, wriggers1995finite}. Two key aspects must be 
addressed in the numerical computation of contact problems in structural mechanics: 
the discretization of the contact surface and the treatment of contact constraints. 
The discretization of the contact surface involves meshing the contact bodies and 
subsequently searching for and matching neighboring entities (nodes, edges, or faces) 
on both sides of the contact interface~\cite{popp2012mortar, STSpapadopoulos1992mixed}. 
The mortar method, introduced in 1994 as a domain decomposition 
technique~\cite{MORbernardi1994new}, is particularly effective in handling mismatched 
interface meshes and is regarded as one of the most advanced methods for contact 
mechanics computation~\cite{MORbelgacem1999mortar, MORtur2009mortar, popp2012mortar}.

Among the methods for imposing contact constraints, the most prominent are the penalty 
method~\cite{PENALTYzavarise2009modified} and the Lagrange multiplier 
method~\cite{LAGpapadopoulos1998lagrange}. The penalty method has the advantage of 
maintaining the size of the discrete linear system similar to that without contact. 
However, it requires the user to define a penalty factor, which can significantly impact 
the convergence and stability of the numerical solution. In structural mechanics 
simulations involving multiple materials and contact pairs, selecting an appropriate 
penalty factor is often challenging. An improper choice can result in a severely ill-
conditioned linear system, making it difficult to solve. Consequently, there is a 
pressing need for efficient solvers tailored for contact problems, particularly for 
saddle-point problems. The Lagrange multiplier method accurately enforces constraints, 
with Lagrange multipliers having a clear physical meaning, representing the contact 
forces on the surface. Thus, it provides a more natural way to impose contact 
constraints. However, it introduces additional degrees of freedom, resulting in a larger
linear system with a saddle-point structure that is symmetric and indefinite, posing 
significant challenges in solving the equations.

While direct methods apply to solving saddle point systems, they often become
impractical for large-scale problems due to significant computational costs and memory 
demands. In contrast, iterative methods with preconditioning
offer a more effective solution for large, sparse linear systems. To tackle saddle point
problems in contact mechanics, researchers have developed block preconditioners
~\cite{INTERNODES2022,liu2013new,franceschini2019block,franceschini2022reverse}, 
whose convergence relies on the properties of the sub-block matrices. Beyond block 
preconditioners, multigrid methods have been successfully employed across various 
fields, such as the Stokes equation~\cite{janka2008smoothed} and incompressible Navier-Stokes 
problems~\cite{griebel1998algebraic}. Adams was the first to apply the algebraic
multigrid (AMG) method to mortar finite element discretized contact 
problems~\cite{adams2004algebraic}, performing standard aggregation on the auxiliary 
matrix graph of the Lagrange multipliers. Building on this, Wiesner et al. proposed AMG preconditioned Krylov subspace methods for solving saddle point systems in contact 
mechanics~\cite{AMGwiesner2021algebraic}, introducing a novel Lagrange multiplier 
aggregation technique on the contact interface to ensure that the saddle point structure
of the matrix is maintained across all coarse grid levels. Although Wiesner's work 
improved the solution of contact problems by considering the $2 \times 2$ structure of
the global matrix, it did not further leverage the specific characteristics of the 
mortar matrix in these problems.

This paper investigates the solution of saddle point equations derived from the mortar
method discretization of contact problems. To simplify the analysis, a two-dimensional 
binding contact problem is utilized as a representative case to explore solution methods
for the corresponding linear systems. Taking into account application features and 
matrix structure properties, this study develops a two-level preconditioning method. This 
method utilizes physical quantities for coarsening and introduces two types of 
interpolation operators, along with corresponding smoothing algorithms, to achieve error
elimination. The coarse grid matrix constructed exhibits symmetry and positive 
definiteness, adequately representing the contact constraints. Furthermore, an 
approximation strategy is introduced to enhance the computational efficiency of the 
method. In comparison to commonly used preconditioning methods, the two-level method 
demonstrates advantages across various contact models, greatly facilitating the solution
process of saddle point equations.

\section{Tied contact problem and FEM discretization}
\label{sect:tie-contact}

In this section, the tied contact models, including the differential form and variational form, are introduced.
For tied contact problem, the contact bodies are tied together
by some shared surfaces.
Taking two bodies as an example, two elastic bodies are denoted as $\Omega^{(1)}$ and $\Omega^{(2)}$, where $\Omega^{(1)}$ is a salve body and $\Omega^{(2)}$ is a master body,
and $\Omega^{(i)} \subset \mathbb{R}^d$, $d = 2$ or 3.
In each subdomain $\Omega^{(i)}$, the linear elastic boundary value problem has the following form~\cite{popp2012mortar}
\begin{equation}
\label{strong}
	\begin{aligned}
		-\nabla \cdot \boldsymbol{\sigma}^ {(i)}  &= \fvec{f}^{(i)}  & \; \text{ in } \; \Omega^{(i)}, \\
		\fvec{u}^{(i)} 			  &= \hat{\fvec{u}}^ {(i)} & \;  \text{ on } \; \Gamma_{\mathrm{u}}^{(i)}, \\
		\boldsymbol{\sigma}^{(i)} \cdot \bn^ {(i)} &= \fvec{t}^{(i)} & \; \text{ on } \; \Gamma_{\sigma}^{(i)},
	\end{aligned}
\end{equation}
which consists of the equilibrium equation,
displacement boundary conditions, and stress boundary conditions.
Herein, $\fvec{u}^{(i)}$ is the displacement,
$\boldsymbol{\sigma}^{(i)}$ is the stress that is dependent on the displacement,
$\fvec{f}^{(i)}$ is the body force, $\bn^ {(i)}$ is the unit outer normal vector
of the domain $\Omega^{(i)}$.
$\Gamma_{\mathrm{u}}^{(i)}$ denotes the Dirichlet boundary, where the displacements
are prescribed by $\hat{\fvec{u}}^ {(i)}$, and $\Gamma_{\sigma}^{(i)}$ represents the Neumann boundary, where the tractions are given by $\fvec{t}^{(i)}$.
For the tied contact problem,
the displacement of each body on the contact surface is equal.
In other words, it can be stated as
\begin{equation}
	\fvec{u}^{(1)} = \fvec{u}^{(2)} \quad  \text{on} \quad \Gamma_{\mathrm{c}} ,
\end{equation}
where $\Gamma_{\mathrm{c}}$ represents the contact boundary of
$\Omega^{(1)}$ and $\Omega^{(2)}$.

To define the weak form of the tied contact problem, we use the function spaces 
$\boldsymbol{\mathcal{U}}^{(i)}$ and $\boldsymbol{\mathcal{V}}^{(i)}$ to represent the 
trial and test function spaces, respectively. Furthermore, the Lagrange multiplier is 
defined in the space $\boldsymbol{\mathcal{M}}$, which is the dual space of the trace 
space of $\boldsymbol{\mathcal{V}}^{(1)}$ on $\Gamma_{\mathrm{c}}$. Consequently, the 
following saddle-point weak form is derived: find $\fvec{u}^{(i)} \in \boldsymbol{\mathcal{U}}^{(i)}$ and 
$\boldsymbol{\lambda} \in \boldsymbol{\mathcal{M}}$ such that 
\begin{align}
			-\delta \mathcal{W}_{\mathrm{int}, \mathrm{ext}} \left( \fvec{u}^{(i)}, \delta \fvec{u}^{(i)} \right) + 
			\int_{\Gamma_{\mathrm{c}}} \boldsymbol{\lambda} \cdot\left(\delta \fvec{u}^{(1)} - \delta \fvec{u}^{(2)} \right) \mathrm{d} S = 0,
			& \quad \forall \; \delta \fvec{u}^{(i)} \in \boldsymbol{\mathcal{V}}^{(i)}, \label{eqn:weakform1} \\ 
			\int_{\Gamma_{\mathrm{c}}} \delta \boldsymbol{\lambda} \cdot \left( \fvec{u}^{(1)} - \fvec{u}^{(2)} \right) \mathrm{d} S = 0,
			& \quad \forall \; \delta \boldsymbol{\lambda} \in \boldsymbol{\mathcal{M}}. \label{eqn:weakform2}
\end{align}
Here, $\mathcal{W}{\mathrm{int},\mathrm{ext}}$ denotes the virtual work term arising 
from internal and external forces~\cite{wang2003}. The second term in 
~\eqref{eqn:weakform1} corresponds to the virtual work generated by contact forces, 
denoted by $\mathcal{W}{\mathrm{mt}}$, while ~\eqref{eqn:weakform2} defines the weak 
form associated with the tied contact constraint, denoted by $\mathcal{W}_{\lambda}$.

Using the weak form, the finite element method can be employed to discretize and obtain
a system of linear equations. In contact mechanics simulations, the mortar finite 
element method~\cite{popp2012mortar} is widely used to handle the coupling terms between
displacement and Lagrange multipliers at the contact interface ($\mathcal{W}
{\mathrm{mt}}$ and $\mathcal{W}{\lambda}$). At the same time, the conventional virtual 
work terms $\mathcal{W}_{\mathrm{int},\mathrm{ext}}$ are discretized using the classical finite element method~\cite{wang2003}.

For instance, in the tied contact problem between two bodies, Figure~\ref{mesh}~
illustrates the discretization under conditions of mesh mismatch. The contact surface 
$\Gamma_{\mathrm{c},h}^{(1)}$ in body $\Omega_{h}^{(1)}$ is designated as the slave 
surface, while the contact surface $\Gamma_{\mathrm{c},h}^{(2)}$ in body 
$\Omega_{h}^{(2)}$ is designated as the master surface.

\begin{figure}[htbp]
	\centering
	\includegraphics[scale=0.6]{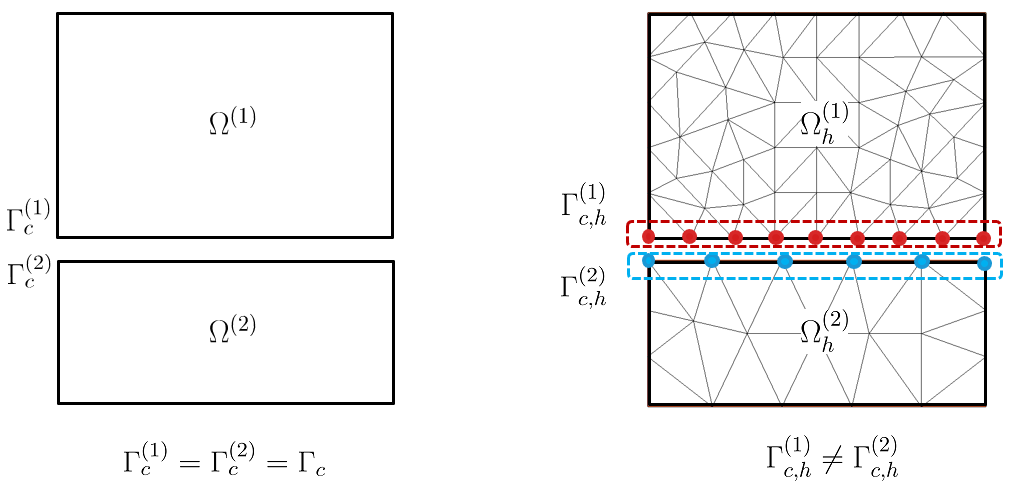}
	\caption{2D two-body tied contact model.}
	\label{mesh}
\end{figure}

The discrete form of the virtual work $\delta \mathcal{W}_{\mathrm{mt}}$ is expressed as
\begin{eqnarray}
\label{mt}
\begin{aligned}
-\delta \mathcal{W}_{\mathrm{mt}, h}
&=
\sum_{j=1}^{m^{(1)}} \sum_{k=1}^{n^{(1)}} \boldsymbol{\lambda}_{j}^{\top} \left(\int_{\Gamma_{\mathrm{c}, h}^{(1)}} \Phi_{j} N_{k}^{(1)} \mathrm{d} A_{0} \right) \delta \bvec{d}_{k}^{(1)} \\
&-\sum_{j=1}^{m^{(1)}} \sum_{l=1}^{n^{(2)}} \boldsymbol{\lambda}_{j}^{\top} \left( \int_{\Gamma_{\mathrm{c}, h}^{(1)}} \Phi_{j} \left( N_{l}^{(2)} \circ \chi_{h} \right) \mathrm{d} A_{0} \right) \delta \bvec{d}_{l}^{(2)},
\end{aligned}
\end{eqnarray}
where $\chi_{h}: \Gamma_{\mathrm{c},h}^{(1)} \rightarrow \Gamma_{\mathrm{c},h}^{(2)}$ is a discrete mapping from slave side to master side.
Note that $\Gamma_{\mathrm{c}}^{(1)} = \Gamma_{\mathrm{c}}^{(2)}$
in continuous case. However, in discrete case, $\Gamma_{\mathrm{c},h}^{(1)} \neq \Gamma_{\mathrm{c},h}^{(2)}$ because of non-matching mesh on contact surfaces.
Since the numerical integration is only performed on the slave side $\Gamma_{\mathrm{c},h}^{(1)}$,
the map $\chi_{h}$ has a crucial influence on the computation~\cite{popp2012mortar}.

Now define
a $m^{(1)} \times n^{(1)}$ block matrix $\bmat{D}$ and
a $m^{(1)} \times n^{(2)}$ block matrix $\bmat{M}$
by
\begin{eqnarray}
		\begin{aligned}
			&{\bmat{D}}[j, k] = D_{j k} {\bmat{I}}_{\textsf{ndim}},  &j=1, \ldots, m^{(1)}, \quad k= 1, \dots , n^{(1)}\\
			&{\bmat{M}}[j, l] = M_{j l} {\bmat{I}}_{\textsf{ndim}},  &j=1, \ldots, m^{(1)}, \quad l = 1, \dots , n^{(2)}.
		\end{aligned}
\end{eqnarray}
where
\begin{eqnarray*}
D_{j k} = \int_{\Gamma_{\mathrm{c}, h}^{(1)}} \Phi_{j} N_{k}^{(1)} \mathrm{d} S,
\quad
M_{j l} = \int_{\Gamma_{\mathrm{c}, h}^{(1)}} \Phi_{j} \left( N_{l}^{(2)} \circ \chi_{h} \right) \mathrm{d} S,
\end{eqnarray*}
and ${\bmat{I}}_{\textsf{ndim}}$ is the ${\textsf{ndim}} \times {\textsf{ndim}}$ identity matrix.
In general, both $\bmat{D}$ and $\bmat{M}$ are not square matrices.
However, $\bmat{D}$ becomes a square matrix for the common choice $m^{(1)} = n^{(1)}$.

Divide the grid nodes into two sets: one is the node set of the slave body, which includes the slave contact surface node set $\mathcal{S}$ and the internal node set $\mathcal{N}_1$; the other is the node set of the master body, which includes the master contact surface node set $\mathcal{M}$ and the internal node set $\mathcal{N}_2$.
Thus, all displacement degrees of freedom can be represented in a block form as
\begin{equation}
    \bvec{d} = [\bvec{d}_1, \bvec{d}_2], \quad
    \bvec{d}_1 = [\bvec{d}_{\mathcal{N}_1}, \bvec{d}_\mathcal{S}], \quad
    \bvec{d}_2 = [\bvec{d}_{\mathcal{N}_2}, \bvec{d}_\mathcal{M}].
\end{equation}
Correspondingly, by the definition of the matrices
$\bmat{D}$ and $\bmat{M}$,
equation~\eqref{mt} can be written in matrix form:
\begin{eqnarray}
\label{mat2}
-\delta \mathcal{W}_{\mathrm{mt}, h}
			&=& \delta \bvec{d}_\mathcal{S}^{\top} {\bmat{D}}^{\top} \boldsymbol{\lambda} - \delta \bvec{d}_\mathcal{M}^{\top} {\bmat{M}}^{\top} \boldsymbol{\lambda}
			= \delta \bvec{d}^{\top}
			\left[
			\begin{array}{c}
				\mathbf{0} \\ {\bmat{D}}^{\top} \\ \mathbf{0} \\ -{\bmat{M}}^{\top}
			\end{array}
			\right] \boldsymbol{\lambda}  \nonumber \\
			&=& \delta \bvec{d}_1^{\top} \tilde{{\bmat{D}}}^{\top} \boldsymbol{\lambda} - \delta \bvec{d}_2^{\top} \tilde{{\bmat{M}}}^{\top} \boldsymbol{\lambda}
			= \delta \bvec{d}^{\top}
			\left[
			\begin{array}{c}
				\tilde{{\bmat{D}}}^{\top} \\ -\tilde{{\bmat{M}}}^{\top}
			\end{array} \right] \boldsymbol{\lambda}
			= \delta \bvec{d}^{\top} {\bmat{G}}^{\top} \boldsymbol{\lambda} ,
\end{eqnarray}
where
\begin{equation}
   {\bmat{G}} = [\tilde{{\bmat{D}}}, -\tilde{{\bmat{M}}}], \quad
    \tilde{{\bmat{D}}} = [\mathbf{0}, {\bmat{D}}], \quad
    \tilde{{\bmat{M}}} = [\mathbf{0}, {\bmat{M}}]. 
\end{equation}
Furthermore, the discrete form of the weak formulation $\mathcal{W}_{\lambda} $ can be expressed as
\begin{eqnarray}
	\label{mat3}
		\begin{aligned}
			\delta \mathcal{W}_{\lambda, h}
			&= \delta \boldsymbol{\lambda}^{\top} {\bmat{D}} \bvec{d}_\mathcal{S} - \delta \boldsymbol{\lambda}^{\top} {\bmat{M}} \bvec{d}_\mathcal{M}
			= \delta \boldsymbol{\lambda}^{\top} \tilde{{\bmat{M}}}  \bvec{d}_1 - \delta \boldsymbol{\lambda}^{\top} \tilde{{\bmat{M}}} \mathrm{d}_2
			= \delta \boldsymbol{\lambda}^{\top} {\bmat{G}} \bvec{d}.
		\end{aligned}
\end{eqnarray}

In summary, the discrete system can be written as follows:
\begin{eqnarray}
	\label{eqn: twobody-disc}
	\begin{aligned}
		{\bmat{K}} \bvec{d} + {\bmat{G}}^{\top} \boldsymbol{\lambda} -\bvec{f}_{\mathrm{ext}} &= 0, \\
		{\bmat{G}} \bvec{d} &= 0,
	\end{aligned}
\end{eqnarray}
which can be expressed compactly by matrix-vector form
\begin{equation}
	\label{eqn: total}
	\left[ \begin{array}{c c }
		\bmat{K} & \bmat{G}^{\top} \\
		\bmat{G} & \mathbf{0}
	\end{array} \right]
	\left[ \begin{array}{c}
		\bvec{d} \\
		\boldsymbol{\lambda}
	\end{array} \right]
	=
	\left[ \begin{array}{c}
		\bvec{f} \\
		\mathbf{0}
	\end{array} \right],
\end{equation}
or succinctly expressed as 
\begin{eqnarray}
\label{eqn:lin-system-Ax-b}    
  \bmat{\mathcal{A}} \bvec{x} = \bvec{b}.
\end{eqnarray}

\section{Two-Level AMG preconditioning method for contact computation}
\label{sect:two-level-amg}
\def\Pideal{\hat{\mathcal{P}}}
\def\Rideal{\hat{\mathcal{R}}}

The AMG method is efficient for solving linear systems in elasticity mechanics due to the 
symmetric positive definiteness of the coefficient matrix. However, in contact mechanics, 
constraints imposed by Lagrange multipliers lead to a saddle point system. Classical AMG 
methods often show slow convergence or fail to converge when applied directly to these 
systems. This is due to the indefinite nature of the saddle point matrix, which couples 
displacement and Lagrange multiplier degrees of freedom representing contact constraints. 
Contact constraints are crucial in contact problems, however, classical AMG methods often 
struggle with the coupling between these different unknowns. Therefore, we propose a 
two-level method tailored specifically to solve saddle point systems in contact 
mechanics. This method aims to design a coarse grid operator that represents contact 
constraints and is easy to solve.

\subsection{Two-Level AMG framework}
\label{subsect:framework}
When computing contact problems, imposing constraints with the Lagrange multiplier method 
results in a saddle point linear system, as illustrated in~\eqref{eqn: total}, 
\begin{equation}
	\left[ \begin{array}{c c }
		\bmat{K} & \bmat{G}^{\top} \\
		\bmat{G} & \mathbf{0}
	\end{array} \right]
	\left[ \begin{array}{c}
		\bvec{d} \\
		\boldsymbol{\lambda}
	\end{array} \right]
	=
	\left[ \begin{array}{c}
		\bvec{f} \\
		\mathbf{0}
	\end{array} \right],
\end{equation}
here, the matrix $\bmat{G}$ represents the algebraic formulation of the contact 
constraints, capturing the coupling between the displacement and the Lagrange 
multipliers. When the displacement degrees of freedom (DOFs) are subdivided into 
$\bvec{d} = \left[ \bvec{d}_{\mathcal{N}},\bvec{d}_{\mathcal{M}},\bvec{d}_{\mathcal{S}} \right] ^{\top}$, 
where $\mathcal{M}$ represents the master surface nodes, $\mathcal{S}$ represents the 
slave surface nodes and $\mathcal{N}$ represents the other nodes, the linear system can 
be equivalently represented as follows: 
\begin{equation}
	\label{eqn:twobody1}
		\left[ \begin{array}{c c c c}
			\bmat{K}_{\mathcal{N} \mathcal{N}} & \bmat{K}_{\mathcal{N} \mathcal{M}} & \bmat{K}_{\mathcal{N} \mathcal{S}}  & \mathbf{0} \\
			\bmat{K}_{\mathcal{M} \mathcal{N}} & \bmat{K}_{\mathcal{M} \mathcal{M}} & \mathbf{0} 							 & -\bmat{M}^{\top} \\
			\bmat{K}_{\mathcal{S} \mathcal{N}} & \mathbf{0} 						  & \bmat{K}_{\mathcal{S} \mathcal{S}}  & \bmat{D}^{\top}  \\
			\mathbf{0} 							& -\bmat{M} 						  & \bmat{D} 							 & \mathbf{0}
		\end{array} \right]
		\left[ \begin{array}{c}
			\bvec{d}_{\mathcal{N}} \\ 
			\bvec{d}_{\mathcal{M}} \\ 
			\bvec{d}_{\mathcal{S}} \\ 
			\boldsymbol{\lambda}
		\end{array} \right]
		=
		\left[ \begin{array}{c}
			\bvec{f}_{\mathcal{N}} \\
			\bvec{f}_{\mathcal{M}} \\
			\bvec{f}_{\mathcal{S}}\\
			\mathbf{0}
		\end{array} \right].
\end{equation}
In designing the two-level method, we perform coarsening based on physical quantities. 
We select the displacement DOFs for the master surface nodes $\bvec{d}\mathcal{M}$, 
and for the other nodes $\bvec{d}\mathcal{N}$, as the coarse set, denoted $C$. The 
remaining unknowns, $\bvec{d}_\mathcal{S}$ and $\boldsymbol{\lambda}$, are treated as 
fine points, denoted $F$: 
\begin{align}\label{eqn:CF-split}
    C=\left\{ \bvec{d}_{\mathcal{N}}, \bvec{d}_{\mathcal{M}} \right\} , \quad
    F=\left\{ \bvec{d}_{\mathcal{S}}, \boldsymbol{\lambda} \right\}.
\end{align}
Rearranging based on coarse and fine points, equation~\eqref{eqn:twobody1} can be expressed as:
\begin{equation}
	\label{eqn:CF}
		\left[ \begin{array}{c c}
                \bmat{A}_{CC} & \bmat{A}_{CF} \\
                \bmat{A}_{FC} & \bmat{A}_{FF}
		\end{array} \right]
		\left[ \begin{array}{c}
    		  \bvec{x}_{C}  \\ 
			\bvec{x}_{F}
		\end{array} \right]
		=
		\left[ \begin{array}{c}
    		  \bvec{b}_{C}  \\ 
			\bvec{b}_{F}
		\end{array} \right],
\end{equation}
where the specific forms of the subblocks are as follows: 
\begin{align}
    \label{eqn:CF-block}
    \bmat{A}_{CC}
    &=
    \left[ \begin{array}{c c}
        \bmat{K}_{\mathcal{N} \mathcal{N}} & \bmat{K}_{\mathcal{N} \mathcal{M}} \\
        \bmat{K}_{\mathcal{M} \mathcal{N}} & \bmat{K}_{\mathcal{M} \mathcal{M}}
    \end{array} \right], \\
    \bmat{A}_{CF} 
    &=
    \left[ \begin{array}{c c}
        \bmat{K}_{\mathcal{N} \mathcal{S}}  & \mathbf{0} \\
        \mathbf{0} 							 & -\bmat{M}^{\top}
    \end{array} \right], \quad 
    \bmat{A}_{FC} = \bmat{A}_{CF}^\top, \\
    \bmat{A}_{FF} 
    &=
    \left[ \begin{array}{c c}
        \bmat{K}_{\mathcal{S} \mathcal{S}}  & \bmat{D}^{\top} \\
        \bmat{D} & \mathbf{0}
    \end{array} \right],
\end{align}
and the subblock $\bmat{A}_{FF}$ resulting from this CF splitting is invertible, 
\begin{align}
    \bmat{A}_{FF}^{-1} = \left[\begin{array}{cc}
        \mathbf{0}       & \bmat{D}^{-1} \\
        \bmat{D}^{-\top} & -\bmat{D}^{-\top} \bmat{K}_{\mathcal{S} \mathcal{S}} \bmat{D}^{-1} 
    \end{array}\right].
\end{align}

Furthermore, let $\bmat{I}_{CC}$ denote the identity matrix in the coarse grid space, 
with the following block form: 
\begin{align}
    \bmat{I}_{CC} = \left[\begin{array}{cc}
        \bmat{I}_{\mathcal{N} \mathcal{N}} & \mathbf{0} \\
        \mathbf{0} & \bmat{I}_{\mathcal{M} \mathcal{M}}
    \end{array}\right],
\end{align}
Let $\mathcal{P}: C \to \Omega$ denote the interpolation operator and 
$\mathcal{R}: \Omega \to C$ denote the restriction operator, with their block forms as follows
\begin{align}\label{eqn:interp-P-R}
    \mathcal{P} = \left[ \begin{array}{c}
        \bmat{I}_{CC} \\
        \bmat{P}_{FC} 
    \end{array}\right], \quad
    \mathcal{R} = \left[ \begin{array}{cc}
        \bmat{I}_{CC} & \bmat{R}_{CF} 
    \end{array}\right].    
\end{align}
Finally, let $\mathcal{A}_{H}$ denote the coarse grid operator.


Based on the definitions provided and the classical two-level method, the two-level 
algorithm for contact problems is outlined in Algorithm~\ref{algorithm:solve-contact}, 
which includes the setup and solve phases.
\begin{algorithm}
	\caption{Two-Level Algorithm Framework for Contact Problems} 
	\label{algorithm:solve-contact}
	\begin{algorithmic}[1]
        \State \textbf{Setup Phase}
        \State \quad Coarsening: Select $\bvec{d}_{\mathcal{M}}$ and $\bvec{d}_{\mathcal{N}}$ as coarse grid points.
        \State \quad Construct the interpolation operator $\mathcal{P}: C \rightarrow \Omega$ and restriction operator $\mathcal{R}: \Omega \rightarrow C$, as given in equation~\eqref{eqn:interp-P-R}.
        \State \quad Construct the coarse grid matrix $\mathcal{A}_{H}$.
        \State \textbf{Solve Phase}
        \State \quad Pre-smoothing: If applied, perform $\mu_1$ smoothing iterations on the fine grid system $\mathcal{A}\bvec{x} = \bvec{b}$ to obtain $\bvec{x}_h$; otherwise, set $\bvec{x}_h = \bvec{0}$.
        \State \quad Compute the residual and restrict to the coarse grid: $\bvec{f}_{h} = \bvec{b} - \mathcal{A}\bvec{x}_{h}$ and $\bvec{f}_{H} = \mathcal{R} \bvec{r}_{h}$.
        \State \quad Solve the coarse grid equation: $\mathcal{A}_{H} \bvec{e}_{H} = \bvec{f}_{H}$.
        \State \quad Interpolate and correct the fine grid approximation: $\bvec{e}_{h} = \mathcal{P}\bvec{e}_{H}$ and update $\bvec{x}_h = \bvec{x}_h + \bvec{e}_{h} $.
        \State \quad Post-smoothing: If applied, perform $\mu_2$ smoothing iterations on the fine grid system $\mathcal{A}\bvec{x} = \bvec{b}$ to update $\bvec{x}_h$; otherwise, keep $\bvec{x}_h$ unchanged.
	\end{algorithmic}
\end{algorithm}

A fundamental principle of the multigrid algorithm is the complementarity between the 
smoothing process and the coarse grid correction process. In geometric multigrid methods, 
the coarse grid correction operator is typically defined first, followed by the 
subsequent design of the smoother to reduce residual error components. Conversely, in 
algebraic multigrid methods, the smoother is often specified first, after which the 
interpolation and coarse grid operators are established based on the setup 
algorithm~\cite{brannick2018optimal}.

In the context of contact problems, the development of the two-level algebraic algorithm
components begins with defining the interpolation and coarse grid operators, followed by 
the design of the smoother. This ensures that the coarse grid correction and smoothing 
processes are effectively complementary, thus enabling error reduction.


\subsection{Interpolation}
\label{subsect:interpolation}

The block decomposition of the coefficient matrix $\mathcal{A}$ in~\eqref{eqn:CF} is as 
follows: 
\begin{align} \label{eqn:A-block-LDU}
    \left[ \begin{array}{c c}
        \bmat{A}_{CC} & \bmat{A}_{CF} \\
        \bmat{A}_{FC} & \bmat{A}_{FF}
    \end{array} \right] 
    =
    \left[ \begin{array}{c c}
        \bmat{I}   & \bmat{A}_{CF} \bmat{A}_{FF}^{-1} \\
        \mathbf{0} & \bmat{I}
    \end{array} \right] 
    \left[ \begin{array}{c c}
        \bmat{S}      & \mathbf{0} \\
        \mathbf{0}    & \bmat{A}_{FF}
    \end{array} \right] 
    \left[ \begin{array}{c c}
        \bmat{I}                            & \mathbf{0} \\
        \bmat{A}_{FF}^{-1} \bmat{A}_{FC}    & \bmat{I}
    \end{array} \right],     
\end{align}
where $\bmat{S} = \bmat{A}_{CC} - \bmat{A}_{CF} \bmat{A}_{FF}^{-1} \bmat{A}_{FC}$ 
represents the Schur complement of $\bmat{A}_{FF}$ within $\mathcal{A}$.

Based on the above decomposition, the ideal interpolation operator $\Pideal$ is defined 
as~\cite{brannick2010compatible, brandt2015bootstrap, wiesner2014multigrid, bui2020scalable, bui2021multigrid}: 
\begin{align}
    \Pideal = 
        \left[ \begin{array}{c}
             \bmat{I}_{CC} \\
             -\bmat{A}_{FF}^{-1} \bmat{A}_{FC} 
        \end{array}\right]
        =
	\left[ \begin{array}{c c}
            \bmat{I}_{\mathcal{N} \mathcal{N}} & \mathbf{0} \\
            \mathbf{0} & \bmat{I}_{\mathcal{M} \mathcal{M}} \\
            \mathbf{0} & \bmat{P} \\
            -\bmat{D}^{-\top} \bmat{K}_{\mathcal{S} \mathcal{N}} & -\bmat{D}^{-\top} \bmat{K}_{\mathcal{S} \mathcal{S}} \bmat{P}
	\end{array} \right], \label{eqn:ideal-P}
\end{align}
where $\bmat{P} = \bmat{D}^{-1} \bmat{M}$.
The restriction operator $\Rideal$ is defined as the transpose of the ideal interpolation operator $\Pideal$,
\begin{align}
    \Rideal = \Pideal^{\top} = 
        \left[ \begin{array}{c c}
            \bmat{I}_{CC}  & -\bmat{A}_{CF} \bmat{A}_{FF}^{-1} 
        \end{array} \right]
	=
	\left[ \begin{array}{c c c c}
	\bmat{I}_{\mathcal{N} \mathcal{N}} & \mathbf{0} 	&\mathbf{0} 	 & -\bmat{K}_{\mathcal{N}\mathcal{S}}\bmat{D}^{-1}	\\
	\mathbf{0} & \bmat{I}_{\mathcal{M} \mathcal{M}}  & \bmat{P}^\top      & -\bmat{P}^\top \bmat{K}_{\mathcal{S} \mathcal{S}}\bmat{D}^{-1}
	\end{array} \right],  \label{eqn:ideal-R}
\end{align}
accordingly, $\Rideal$ is termed the ideal restriction operator. It is important to note 
that constructing the ideal interpolation operator and the ideal restriction operator 
requires the computation of $\bmat{A}_{FF}^{-1}$. However, in the context of contact 
problems, this step only requires the computation of $\bmat{D}^{-1}$, as discussed in 
Section~\ref{subsect:design}.

Furthermore, the coarse grid operator (Galerkin type) is defined as 
$\hat{\mathcal{A}}_{H} = \hat{\mathcal{R}} \mathcal{A} \hat{\mathcal{P}} =  \bmat{A}_{CC} - \bmat{A}_{CF} \bmat{A}_{FF}^{-1} \bmat{A}_{FC} = \bmat{S}$, 
which represents the Schur complement matrix of the original saddle point system. The 
block form is as follows:
\begin{align} \label{eqn:ideal-coarse}
        \hat{\mathcal{A}}_{H}
        = \hat{\mathcal{R}} \bmat{\mathcal{A}} \hat{\mathcal{P}} = 
	\left[ \begin{array}{c c}
		\bmat{K}_{\mathcal{N} \mathcal{N}} & \bmat{K}_{\mathcal{N} \mathcal{M}} + \bmat{K}_{\mathcal{N} \mathcal{S}} \bmat{P} \\
		\bmat{K}_{\mathcal{M} \mathcal{N}} + \bmat{P}^{\top} \bmat{K}_{\mathcal{S} \mathcal{N}} & \bmat{K}_{\mathcal{M} \mathcal{M}} + \bmat{P}^{\top} \bmat{K}_{\mathcal{S} \mathcal{S}} \bmat{P}
	\end{array} \right].
\end{align}
Given the specific forms of the interpolation operator, restriction operator, and coarse 
grid operator established, a two-level iterative algorithm can be developed, as 
illustrated in algorithm~\ref{algorithm:solve-strategy-one}.
\begin{algorithm}
	\caption{Two-Level Iterative Algorithm}
	\label{algorithm:solve-strategy-one}
	\hspace*{0.02in} {\bf Input:} Matrix $\mathcal{A}$, right-hand side vector $\bvec{b}$, initial solution $\bvec{x}^{(0)}$, smoother $\bmat{\mathcal{B}}$\\
	\hspace*{0.02in} {\bf Output:}  Iterative solution $\bvec{x}^{(k+1)}$
	\begin{algorithmic}[1]
        \State For \, $k = 0, 1, 2, \cdots$ until convergence
        \State \quad Pre-smoothing: $\bvec{x}_{s} = \bvec{x}^{(k)} + \bmat{\mathcal{B}}^{-1} \left(\bvec{b} - \mathcal{A} \bvec{x}^{(k)} \right) $
	\State \quad Compute residual: $\bvec{f} = \bvec{b} - \mathcal{A}\bvec{x}_{s}$
	\State \quad Restrict:  $\bvec{f}_{H} = \hat{\mathcal{R}} \bvec{f}$
	\State \quad Solve the coarse grid system: $\hat{\mathcal{A}}_{H} \bvec{e}_{H} = \bvec{f}_{H}$, 
                     where $\hat{\mathcal{A}}_{H} = \hat{\mathcal{R}} \mathcal{A} \hat{\mathcal{P}}$
        \State \quad Interpolate: $\bvec{e} = \hat{\mathcal{P}} \bvec{e}_{H}$
        \State \quad Correct: $\bvec{x}^{(k+1)} = \bvec{x}_{s} + \bvec{e}$
        \State EndFor
	\end{algorithmic}
\end{algorithm}

Next, we analyze the form of the smoother that corresponds to the ideal interpolation 
operator. We begin by presenting Proposition~\ref{prop:strategy-one}.

\begin{proposition}\label{prop:strategy-one}
    Using algorithm~\ref{algorithm:solve-strategy-one} to solve the linear system 
    $\mathcal{A} \bvec{x} = \bvec{b}$, with interpolation defined by 
    $\bvec{e} = \hat{\mathcal{P}} \bvec{e}_{H}$, and for any initial value 
    $\bvec{x}^{(0)}$ and smoother $\mathcal{B}$, let the smoothed vector be 
    $\bvec{x}_{s} = [ \bvec{x}_{s,C}, \bvec{x}_{s,F} ]^{\top}$. If the coarse grid system 
    $\hat{\mathcal{A}}_{H}\bvec{e}_{H} = \bvec{f}_{H}$ is solved exactly, the residual 
    after one iteration is
    \begin{align} \label{eqn:residual}
        \bvec{r} = \bvec{b} - \mathcal{A} \bvec{x}^{(1)} =  
        \left[ \begin{array}{c}
            \bmat{A}_{CF} \bmat{A}_{FF}^{-1} \bvec{b}_{F} - \bmat{A}_{CF} \bvec{x}_{s,F} - \bmat{A}_{CF} \bmat{A}_{FF}^{-1} \bmat{A}_{FC} \bvec{x}_{s,C}\\
            \bvec{b}_{F} - \bmat{A}_{FF} \bvec{x}_{s,F} - \bmat{A}_{FC} \bvec{x}_{s,C}  
        \end{array}\right].
    \end{align}
\end{proposition}
\begin{proof}
    Define the smoothed residual vector as $\bvec{f} = \bvec{b} - \mathcal{A}\bvec{x}_{s}$, 
    where its block representation is given by
    \begin{align*}
        \bvec{f} = 
        \left[ \begin{array}{c}
            \bvec{f}_{C} \\
            \bvec{f}_{F} 
        \end{array}\right]
        = 
        \left[ \begin{array}{c}
            \bvec{b}_{C} \\
            \bvec{b}_{F} 
        \end{array}\right] -
        \left[ \begin{array}{c c}
            \bmat{A}_{CC}  & \bmat{A}_{CF} \\
            \bmat{A}_{FC}  & \bmat{A}_{FF}
        \end{array} \right]
        \left[ \begin{array}{c}
            \bvec{x}_{s,C} \\
            \bvec{x}_{s,F} 
        \end{array}\right] 
        =
        \left[ \begin{array}{c}
            \bvec{b}_{C} - \bmat{A}_{CC}\bvec{x}_{s,C} - \bmat{A}_{CF}\bvec{x}_{s,F}\\
            \bvec{b}_{F} - \bmat{A}_{FC}\bvec{x}_{s,C} - \bmat{A}_{FF}\bvec{x}_{s,F}
        \end{array}\right]  ,
    \end{align*}
    Restricting the residual $\bvec{f}$ to the coarse grid level results in 
    $
    \bvec{f}_{H} = \hat{\mathcal{R}} \bvec{f} = \bvec{f}_{C} - \bmat{A}_{CF} \bmat{A}_{FF}^{-1} \bvec{f}_{F} .
    $
    If the coarse grid linear system is solved exactly, then $\bvec{e}_{H} = \mathcal{A}_{H}^{-1} \bvec{f}_{H}$. 
    Subsequent interpolation yields 
    \begin{equation*}
            \bvec{e} = \mathcal{P} \bvec{e}_{H}
            = 
            \left[ \begin{array}{c}
                 \bmat{I}_{CC} \\
                 -\bmat{A}_{FF}^{-1} \bmat{A}_{FC} 
            \end{array}\right]
            \bvec{e}_{H}
            =
            \left[ \begin{array}{c}
                 \bvec{e}_{H} \\
                 -\bmat{A}_{FF}^{-1} \bmat{A}_{FC} \bvec{e}_{H}
            \end{array}\right]. 
    \end{equation*}
    The solution after one iteration is then given by $$\bvec{x}^{(1)} = \bvec{x}_{s} + \bvec{e},$$
    and the corresponding residual is
    \begin{align}\label{eqn:residual-strategy-one}
        \bvec{r} = \bvec{b} - \mathcal{A} \bvec{x}^{(1)} 
        = \bvec{b} - \mathcal{A} (\bvec{x}_{s} + \bvec{e})
        = \bvec{f} - \mathcal{A} \bvec{e},
    \end{align}
    where
    \begin{align*}
        \mathcal{A}\bvec{e} 
        =
            \left[ \begin{array}{c c}
                \bmat{A}_{CC}  & \bmat{A}_{CF} \\
                \bmat{A}_{FC}  & \bmat{A}_{FF}
            \end{array} \right]
            \left[ \begin{array}{c}
                 \bvec{e}_{H} \\
                 -\bmat{A}_{FF}^{-1} \bmat{A}_{FC} \bvec{e}_{H}
            \end{array}\right]  
        =
            \left[ \begin{array}{c}
                 \bmat{A}_{CC} \bvec{e}_{H} - \bmat{A}_{CF} \bmat{A}_{FF}^{-1} \bmat{A}_{FC} \bvec{e}_{H} \\
                 \bmat{A}_{FC} \bvec{e}_{H} - \bmat{A}_{FC} \bvec{e}_{H}
            \end{array}\right]  
        =
            \left[ \begin{array}{c}
                 \bmat{S} \bvec{e}_{H} \\
                 \bvec{0}
            \end{array}\right] .
    \end{align*}
    Given that $\bvec{e}_{H} = \hat{\mathcal{A}}_{H}^{-1} \bvec{f}_{H} = \bmat{S}^{-1} \bvec{f}_{H}$, it follows that
    \begin{align*}
        \mathcal{A}\bvec{e} 
        =
            \left[ \begin{array}{c}
                 \bvec{f}_{H} \\
                 \bvec{0}
            \end{array}\right] .
    \end{align*}
    Furthermore,
    \begin{align*}
        \bvec{f}_{H} = \bvec{f}_{C} - \bmat{A}_{CF} \bmat{A}_{FF}^{-1} \bvec{f}_{F} 
        = \bvec{b}_{C} - \bmat{A}_{CF} \bmat{A}_{FF}^{-1} \bvec{b}_{F} - \bmat{S} \bvec{x}_{s,C},
    \end{align*}
    Substituting this block representation into equation~\eqref{eqn:residual-strategy-one} results in 
    \begin{align}
        \bvec{r} = \bvec{f} - \mathcal{A} \bvec{e}
        &=\left[ \begin{array}{c}
            \bvec{f}_{C} \\
            \bvec{f}_{F} 
        \end{array}\right] 
        -
        \left[ \begin{array}{c}
             \bvec{f}_{H} \\
             \bvec{0}
        \end{array}\right]   \\
        &=
        \left[ \begin{array}{c}
            \bmat{A}_{CF} \bmat{A}_{FF}^{-1} \bvec{b}_{F} - \bmat{A}_{CF} \bvec{x}_{s,F} - \bmat{A}_{CF} \bmat{A}_{FF}^{-1} \bmat{A}_{FC} \bvec{x}_{s,C}\\
            \bvec{b}_{F} - \bmat{A}_{FF} \bvec{x}_{s,F} - \bmat{A}_{FC} \bvec{x}_{s,C}  
        \end{array}\right] .
    \end{align}
\end{proof}

Proposition~\ref{prop:strategy-one} shows that the residual vector after one iteration
can be expressed in terms of the smoothed vector. Our objective in designing the smoother
is to minimize the norm of the residual vector $\bvec{r}$. Accordingly, we propose two 
smoothers that align with the ideal interpolation operator.

\subsection{Smoother} 
\label{subsect:smoother}

\subsubsection{F-relaxation}

The first smoothing method is called F-relaxation, where only the F-points are smoothed. To analyze the effect of smoothing on the residual, we first estimate the 1-norm of the residual vector after one iteration of the two-level method, as given in Equation~\eqref{eqn:residual}:
\begin{align*}
    \| \bvec{r} \| =& \| \bmat{A}_{CF} \bmat{A}_{FF}^{-1} \bvec{b}_{F} - \bmat{A}_{CF} \bvec{x}_{s,F} - \bmat{A}_{CF} \bmat{A}_{FF}^{-1} \bmat{A}_{FC} \bvec{x}_{s,C} \|_{1} + \| \bvec{b}_{F} - \bmat{A}_{FF} \bvec{x}_{s,F} - \bmat{A}_{FC} \bvec{x}_{s,C}  \|_{1} \\ 
    \leq&  \| \bmat{A}_{CF} \bmat{A}_{FF}^{-1} \bvec{b}_{F} - \bmat{A}_{CF} \bvec{x}_{s,F}  \|_{1} + \| \bmat{A}_{CF} \bmat{A}_{FF}^{-1} \bmat{A}_{FC} \bvec{x}_{s,C} \|_{1}  \\ 
     & + \| \bvec{b}_{F} - \bmat{A}_{FF} \bvec{x}_{s,F} \|_{1} + \| \bmat{A}_{FC} \bvec{x}_{s,C} \|_{1} \\
    =& \| \bmat{A}_{CF} \bmat{A}_{FF}^{-1} ( \bvec{b}_{F} - \bmat{A}_{FF} \bvec{x}_{s,F} ) \|_{1} + \| \bvec{b}_{F} - \bmat{A}_{FF} \bvec{x}_{s,F} \|_{1}  \\
     & + \| \bmat{A}_{CF} \bmat{A}_{FF}^{-1} \bmat{A}_{FC} \bvec{x}_{s,C} \|_{1} +  \| \bmat{A}_{FC} \bvec{x}_{s,C} \|_{1}, 
\end{align*}
where $\bvec{x}_{s,C}$ and $\bvec{x}_{s,F}$ represent the components of the solution vector corresponding to the coarse points and fine points after pre-smoothing, respectively.

Given a convergence threshold $\varepsilon$, if only the fine points are smoothed, i.e., $\bvec{x}_{s,C} = 0$, then when $\| \bvec{b}_{F} - \bmat{A}_{FF} \bvec{x}_{s,F} \| < \varepsilon$, the residual of the original system satisfies
\begin{align*}
    \| \bvec{r} \| \leq \| 
    \bmat{A}_{CF} \bmat{A}_{FF}^{-1} \| \varepsilon + \varepsilon = C \varepsilon, \quad C = 1 + \| \bmat{A}_{CF} \bmat{A}_{FF}^{-1} \|.
\end{align*}
We can state the following proposition.

\begin{proposition}\label{prop:F-relaxtion-residual}
    With the restriction operator $\hat{\mathcal{R}}$ and interpolation operator $\hat{\mathcal{P}}$ defined by Equations~\eqref{eqn:ideal-R} and \eqref{eqn:ideal-P}, and the coarse grid operator $\bmat{\mathcal{A}}_{H} = \hat{\mathcal{R}} \mathcal{A} \hat{\mathcal{P}}$ defined by Equation~\eqref{eqn:ideal-coarse}, using Algorithm~\ref{algorithm:solve-strategy-one} to solve the linear system $\mathcal{A} \bvec{x} = \bvec{b}$, where the coarse grid system $\bmat{\mathcal{A}_{H}} \bvec{e}_{H} = \bvec{f}_{H}$ is solved exactly. If only the fine points are smoothed, and $\| \bvec{b}_{F} - \bmat{A}_{FF} \bvec{x}_{s,F} \| < \varepsilon$, then after one iteration of the two-level algorithm~\ref{algorithm:solve-strategy-one}, the residual of the original system satisfies
    \begin{align}
         \| \bvec{r} \| \leq  C \varepsilon, \quad C = 1 + \| \bmat{A}_{CF} \bmat{A}_{FF}^{-1} \|.
    \end{align}
\end{proposition}
Clearly, if the fine points are smoothed exactly, i.e., $\| \bvec{b}_{F} - \bmat{A}_{FF} \bvec{x}_{s,F} \| = 0$, then $\| \bvec{r} \| = 0$, and the two-level method becomes a direct method.

Introducing the matrix $\bmat{\mathcal{Q}}$
\begin{align}\label{eqn:Q}
	\bmat{\mathcal{Q}}
	=
        \left[ \begin{array}{c}
             \bmat{0} \\
             \bmat{I}_{FF} 
        \end{array}\right]
        =
	\left[ \begin{array}{c c}
            \mathbf{0} & \mathbf{0} \\
            \mathbf{0} & \mathbf{0} \\
            \bmat{I}_{\mathcal{S} \mathcal{S}} & \mathbf{0} \\
            \mathbf{0} & \bmat{I}_{\lambda \lambda}
	\end{array} \right], 
\end{align}
the exact F-point smoother $\mathcal{B}_{F}$ can be defined as
\begin{align} \label{eqn:F-relaxtion}
        \mathcal{B}_{F}^{-1} 
        = \mathcal{Q} (\mathcal{Q}^\top \mathcal{A} \mathcal{Q})^{-1} \mathcal{Q}^\top
        =
        \left[ \begin{array}{c c}
            \bmat{0}  & \bmat{0} \\
            \bmat{0}  & \bmat{A}_{FF}^{-1}            
        \end{array} \right] 
        =
		\left[ \begin{array}{c c c c}
			\mathbf{0} 	 & \mathbf{0}  & \mathbf{0}  & \mathbf{0} \\
			\mathbf{0} 	 & \mathbf{0}  & \mathbf{0}  & \mathbf{0} \\
			\mathbf{0} 	 & \mathbf{0}  & \mathbf{0}  & \bmat{D}^{-1} \\
			\mathbf{0} 	 & \mathbf{0}  & \bmat{D}^{-\top}  & -\bmat{D}^{-\top} \bmat{K}_{\mathcal{S} \mathcal{S}} \bmat{D}^{-1}
		\end{array} \right].
\end{align}

\begin{lemma} \label{lemma:add-equal-multi}
    Let the interpolation operator $\mathcal{P}$ and restriction operator $\mathcal{R}$ be defined as 
    \begin{align*}
    \mathcal{P} = \left[ \begin{array}{c}
        \bmat{I}_{CC} \\
        \bmat{P}_{FC} 
    \end{array}\right],
    \quad
    \mathcal{R} = \left[ \begin{array}{cc}
        \bmat{I}_{CC} & \bmat{R}_{CF} 
    \end{array}\right],
    \end{align*}
    and the smoother is $\mathcal{B}^{-1}_{F}$ defined by Equation~\eqref{eqn:F-relaxtion}. Define the two-level preconditioner as
    \begin{align}
        \mathcal{M} 
        =
        \left[ \begin{array}{c c}
            \bmat{I}   & -\bmat{R}_{CF} \\
            \mathbf{0} & \bmat{I}
        \end{array} \right] 
        \left[ \begin{array}{c c}
            \bmat{A}_{H}      & \mathbf{0} \\
            \mathbf{0}    & \bmat{A}_{FF}
        \end{array} \right] 
        \left[ \begin{array}{c c}
            \bmat{I}              & \mathbf{0} \\
            -\bmat{P}_{FC}     & \bmat{I}
        \end{array} \right],    
    \end{align}
    where $\bmat{A}_{H}$ is the coarse grid operator. Then,
    \begin{align} \label{eqn:add}
        \mathcal{I} - \mathcal{M}^{-1} \mathcal{A} = \mathcal{I} - \mathcal{P} \bmat{A}_{H}^{-1} \mathcal{R} \mathcal{A} -
        \mathcal{B}^{-1}_{F} \mathcal{A} .
    \end{align}
    If $\mathcal{R} \mathcal{A} \mathcal{Q} = 0 $, then
    \begin{align}
        \mathcal{I} - \mathcal{M}^{-1} \mathcal{A} = \left( \mathcal{I} - \mathcal{P} \bmat{A}_{H}^{-1} \mathcal{R}\mathcal{A} \right) \left(\mathcal{I} - \mathcal{B}^{-1}_{F} \mathcal{A} \right).  
    \end{align}
    If $\mathcal{Q}^{\top} \mathcal{A} \mathcal{P} = 0 $, then
    \begin{align}
        \mathcal{I} - \mathcal{M}^{-1} \mathcal{A} = \left(\mathcal{I} - \mathcal{B}^{-1}_{F} \mathcal{A} \right) \left( \mathcal{I} - \mathcal{P} \bmat{A}_{H}^{-1} \mathcal{R}\mathcal{A} \right). 
    \end{align}
\end{lemma}

\begin{proof}
    From the block decomposition form of $\mathcal{M}$, it is easy to obtain
    \begin{align} \label{eqn:M-inverse-LDU}
        \mathcal{M}^{-1} = 
        \left[ \begin{array}{c c}
            \bmat{I}              & \mathbf{0} \\
            \bmat{P}_{FC}     & \bmat{I}
        \end{array} \right]   
        \left[ \begin{array}{c c}
            \bmat{A}_{H}^{-1}      & \mathbf{0} \\
            \mathbf{0}    & \bmat{A}_{FF}^{-1}
        \end{array} \right] 
        \left[ \begin{array}{c c}
            \bmat{I}   & \bmat{R}_{CF} \\
            \mathbf{0} & \bmat{I}
        \end{array} \right] ,
    \end{align}
    where the upper and lower triangular blocks can be written as:
    \begin{align*} 
        \left[ \begin{array}{c c}
            \bmat{I}          & \mathbf{0} \\
            \bmat{P}_{FC}     & \bmat{I}
        \end{array} \right]   
        =
        \left[ \begin{array}{c c}
            \mathcal{P}  & \mathcal{Q} 
        \end{array} \right]  ,  \quad
        \left[ \begin{array}{c c}
            \bmat{I}   & \bmat{R}_{CF} \\
            \mathbf{0} & \bmat{I}
        \end{array} \right] 
        =
        \left[ \begin{array}{c}
            \mathcal{R}  \\
            \mathcal{Q}^{\top} 
        \end{array} \right],
    \end{align*}
    Substituting into Equation~\eqref{eqn:M-inverse-LDU} gives
    \begin{align*}
        \mathcal{M}^{-1} = 
        \left[ \begin{array}{c c}
            \mathcal{P}  & \mathcal{Q} 
        \end{array} \right]
        \left[ \begin{array}{c c}
            \bmat{A}_{H}^{-1}      & \mathbf{0} \\
            \mathbf{0}    & \bmat{A}_{FF}^{-1}
        \end{array} \right] 
        \left[ \begin{array}{c}
            \mathcal{R}  \\
            \mathcal{Q}^{\top} 
        \end{array} \right]
        =
        \mathcal{P} \bmat{A}_{H}^{-1} \mathcal{R} +
        \mathcal{Q} \bmat{A}_{FF}^{-1} \mathcal{Q}^{\top}   ,     
    \end{align*}
    thus
    \begin{align*}
        \mathcal{I} - \mathcal{M}^{-1} \mathcal{A} =
        \mathcal{I} - \mathcal{P} \bmat{A}_{H}^{-1} \mathcal{R} \mathcal{A} -
        \mathcal{Q} \bmat{A}_{FF}^{-1} \mathcal{Q}^{\top} \mathcal{A} .
    \end{align*}
    Clearly, if $\mathcal{R} \mathcal{A} \mathcal{Q} = 0 $, then
    \begin{align*}
       \mathcal{I} - \mathcal{P} \bmat{A}_{H}^{-1} \mathcal{R} \mathcal{A} -
        \mathcal{Q} \bmat{A}_{FF}^{-1} \mathcal{Q}^{\top} \mathcal{A} 
        =
        \left( \mathcal{I} - \mathcal{P} \bmat{A}_{H}^{-1} \mathcal{R}\mathcal{A} \right) \left(\mathcal{I} - \mathcal{Q} \bmat{A}_{FF}^{-1} \mathcal{Q}^{\top} \mathcal{A} \right). 
    \end{align*}
    If $\mathcal{Q}^{\top} \mathcal{A} \mathcal{P} = 0 $, then
    \begin{align*}
        \mathcal{I} - \mathcal{P} \bmat{A}_{H}^{-1} \mathcal{R} \mathcal{A} -
        \mathcal{Q} \bmat{A}_{FF}^{-1} \mathcal{Q}^{\top} \mathcal{A} 
        =
        \left(\mathcal{I} - \mathcal{Q} \bmat{A}_{FF}^{-1} \mathcal{Q}^{\top} \mathcal{A} \right) \left( \mathcal{I} - \mathcal{P} \bmat{A}_{H}^{-1} \mathcal{R}\mathcal{A} \right) .
    \end{align*}
\end{proof}

In Lemma~\ref{lemma:add-equal-multi}, the iteration format given by Equation~\eqref{eqn:add} is referred to as the additive two-level method~\cite{wiesner2014multigrid}. This implies that the exact smoothing of the fine points and the coarse grid correction can be performed simultaneously. If the restriction operator $\mathcal{R}$ satisfies $\mathcal{R} \mathcal{A} \mathcal{Q} = 0$, then the additive two-level method can be represented as a two-level method with pre-smoothing; if the interpolation operator $\mathcal{P}$ satisfies $\mathcal{Q}^{\top} \mathcal{A} \mathcal{P} = 0$, then the additive two-level method can be represented as a two-level method with post-smoothing. 

The following analysis considers the ideal restriction operator $\hat{\mathcal{R}}$ and ideal interpolation operator $\hat{\mathcal{P}}$ defined by Equations~\eqref{eqn:ideal-R} and \eqref{eqn:ideal-P}, respectively, and the exact fine-point smoother $\mathcal{B}_{F}$ defined by Equation~\eqref{eqn:F-relaxtion}, to draw conclusions on the convergence of the two-level method.

\begin{proposition}\label{prop:F-relaxtion-direct}
    With the restriction operator $\hat{\mathcal{R}}$ and interpolation operator $\hat{\mathcal{P}}$ defined by Equations~\eqref{eqn:ideal-R} and \eqref{eqn:ideal-P}, and the coarse grid operator defined as $\hat{\mathcal{R}} \mathcal{A} \hat{\mathcal{P}}$, if the coarse grid system is solved exactly and the smoother is $\mathcal{B}_{F}$ as defined in Equation~\eqref{eqn:F-relaxtion}, then the resulting two-level method is a direct method and satisfies
    \begin{align}
        \left(\mathcal{I} - \hat{\mathcal{P}} (\hat{\mathcal{R}} \mathcal{A} \hat{\mathcal{P}})^{-1} \hat{\mathcal{R}} \mathcal{A} \right) \left(\mathcal{I} -\mathcal{B}_{F}^{-1}  \mathcal{A} \right)
        & = \bmat{0},  \\
        \left(\mathcal{I} -\mathcal{B}_{F}^{-1}  \mathcal{A} \right) \left(\mathcal{I} - \hat{\mathcal{P}} (\hat{\mathcal{R}} \mathcal{A} \hat{\mathcal{P}})^{-1} \hat{\mathcal{R}} \mathcal{A} \right)
        & = \bmat{0}.
    \end{align}
\end{proposition}

\begin{proof}
    For the matrix $\mathcal{A}$, if $\bmat{A}_{FF}$ is invertible, then
    \begin{eqnarray*}
        \mathcal{A}^{-1} &=& 
            \left[ \begin{array}{c c}
                \bmat{S}^{-1}  & -\bmat{S}^{-1} \bmat{A}_{CF} \bmat{A}_{FF}^{-1}\\
                -\bmat{A}_{FF}^{-1} \bmat{A}_{FC} \bmat{S}^{-1} & \bmat{A}_{FF}^{-1} \bmat{A}_{FC} \bmat{S}^{-1} \bmat{A}_{CF} \bmat{A}_{FF}^{-1} 
            \end{array} \right] +       
            \left[ \begin{array}{c c}
                \bmat{0}  & \bmat{0}\\
                \bmat{0} & \bmat{A}_{FF}^{-1} 
            \end{array} \right] \\
            &=& 
            \hat{\mathcal{P}} (\hat{\mathcal{R}} \mathcal{A} \hat{\mathcal{P}})^{-1} \hat{\mathcal{R}} + \mathcal{Q} (\mathcal{Q}^\top \mathcal{A} \mathcal{Q})^{-1} \mathcal{Q}^\top,     
    \end{eqnarray*}
    thus
    \begin{align*}
        \bmat{0} = \mathcal{I} - \mathcal{A}^{-1} \mathcal{A} 
        = \mathcal{I} - \hat{\mathcal{P}} (\hat{\mathcal{R}} \mathcal{A} \hat{\mathcal{P}})^{-1} \hat{\mathcal{R}} \mathcal{A} - \mathcal{Q} (\mathcal{Q}^\top \mathcal{A} \mathcal{Q})^{-1} \mathcal{Q}^\top \mathcal{A}.
    \end{align*}
    Noting that $\hat{\mathcal{R}} \mathcal{A} \mathcal{Q} = \mathcal{Q}^\top \mathcal{A} \hat{\mathcal{P}} = \bmat{0}$, the following equalities hold:
    \begin{align*}
        \bmat{0} = \mathcal{I} - \mathcal{A}^{-1} \mathcal{A} 
        &= \left(\mathcal{I} - \hat{\mathcal{P}} (\hat{\mathcal{R}} \mathcal{A} \hat{\mathcal{P}})^{-1} \hat{\mathcal{R}} \mathcal{A} \right)  
        \left(\mathcal{I} - \mathcal{Q} (\mathcal{Q}^\top \mathcal{A} \mathcal{Q})^{-1} \mathcal{Q}^\top \mathcal{A} \right) \\
        &= \left(\mathcal{I} - \mathcal{Q} (\mathcal{Q}^\top \mathcal{A} \mathcal{Q})^{-1} \mathcal{Q}^\top \mathcal{A} \right) 
        \left(\mathcal{I} - \hat{\mathcal{P}} (\hat{\mathcal{R}} \mathcal{A} \hat{\mathcal{P}})^{-1} \hat{\mathcal{R}} \mathcal{A} \right).
    \end{align*}     
\end{proof}

In Propositions~\ref{prop:F-relaxtion-residual} and \ref{prop:F-relaxtion-direct}, it is required that the coarse grid system is solved exactly, which is difficult to achieve in practical applications. Further examining the form of the coarse grid operator $\hat{\mathcal{A}}_{H} = \hat{\mathcal{R}} \mathcal{A} \hat{\mathcal{P}} = \bmat{S}$ in Equation~\eqref{eqn:ideal-coarse},  $\hat{\mathcal{A}}_{H}$ is symmetric positive definite, which allows it to be solved using AMG. Let the mathematical operator for solving the coarse grid system $\hat{\mathcal{A}}_{H} \bvec{e}_{H} = \bvec{f}_{H}$ using AMG be denoted as $\hat{\mathcal{G}}_{H}^{-1}$.

Next, we analyze the effect of performing $\nu$ AMG V-cycles on the coarse grid system $\hat{\mathcal{A}}_{H} \bvec{e}_{H} = \bvec{f}_{H}$ in Algorithm~\ref{algorithm:solve-strategy-one}. If $\nu = 1$, the entire algorithm is a V-cycle, with the only difference being that the coarsening, smoothing, and interpolation at the first level are distinct from those at other levels.

\begin{theorem}\label{theorem:M-F-AMG}
    With the restriction operator $\hat{\mathcal{R}}$ and interpolation operator $\hat{\mathcal{P}}$ defined by Equations~\eqref{eqn:ideal-R} and \eqref{eqn:ideal-P}, and the coarse grid operator defined as $\hat{\mathcal{A}}_{H} = \hat{\mathcal{R}} \mathcal{A} \hat{\mathcal{P}}$, the smoother is $\mathcal{B}_{F}$. When AMG is used to solve the coarse grid system $\hat{\mathcal{A}}_{H} \bvec{e}_{H} = \bvec{f}_{H}$, let the mathematical operator be denoted as $\hat{\mathcal{G}}_{H}^{-1}$. Define
    \begin{align}\label{eqn:M-F-AMG-block-LDU}
        \mathcal{M} 
        =
        \left[ \begin{array}{c c}
            \bmat{I}   & \bmat{A}_{CF} \bmat{A}_{FF}^{-1} \\
            \mathbf{0} & \bmat{I}
        \end{array} \right] 
        \left[ \begin{array}{c c}
            \hat{\mathcal{G}}_{H}      & \mathbf{0} \\
            \mathbf{0}    & \bmat{A}_{FF}
        \end{array} \right] 
        \left[ \begin{array}{c c}
            \bmat{I}              & \mathbf{0} \\
            \bmat{A}_{FF}^{-1} \bmat{A}_{FC}    & \bmat{I}
        \end{array} \right],
    \end{align}
    then
    \begin{align}\label{eqn:M-F-AMG-iter}
        \mathcal{I} - \mathcal{M}^{-1} \mathcal{A} 
        = \left(\mathcal{I} - \hat{\mathcal{P}} \hat{\mathcal{G}}_{H}^{-1} \hat{\mathcal{R}} \mathcal{A} \right) \left(\mathcal{I} -\mathcal{B}_{F}^{-1}  \mathcal{A} \right) 
        = \left(\mathcal{I} -\mathcal{B}_{F}^{-1}  \mathcal{A} \right) \left(\mathcal{I} - \hat{\mathcal{P}} \hat{\mathcal{G}}_{H}^{-1} \hat{\mathcal{R}} \mathcal{A} \right),
    \end{align}
    and the eigenvalue distribution of $\mathcal{M}^{-1} \mathcal{A}$ is given by
    \begin{align}\label{eqn:M-F-AMG-eigenvalue}
        \sigma\left( \mathcal{M}^{-1} \mathcal{A} \right) = \left\{ 1 \right\} \cup  \left\{ \sigma_{j}: \; \sigma_{j} \in \sigma\left( \hat{\mathcal{G}}_{H}^{-1} \bmat{S} \right) \right\},
    \end{align}
    where $\sigma(\bmat{K})$ denotes the spectrum of the matrix $\bmat{K}$, and $\sigma_{j}$ represents an eigenvalue in the set.
\end{theorem}
\begin{proof}
    Noting that $\hat{\mathcal{R}} \mathcal{A} \mathcal{Q} = \mathcal{Q}^\top \mathcal{A} \hat{\mathcal{P}} = \bmat{0}$, by Lemma~\ref{lemma:add-equal-multi}, it is easy to verify Equation~\eqref{eqn:M-F-AMG-iter}. Further, from the block form of matrix $\mathcal{A}$ in Equation~\eqref{eqn:A-block-LDU} and the block form of matrix $\mathcal{M}$ in Equation~\eqref{eqn:M-F-AMG-block-LDU}, we can obtain
    \begin{align*}
        \mathcal{M}^{-1} \mathcal{A} = &
        \left[ \begin{array}{c c}
            \bmat{I}              & \mathbf{0} \\
            \bmat{A}_{FF}^{-1} \bmat{A}_{FC}    & \bmat{I}
        \end{array} \right]^{-1}
        \left[ \begin{array}{c c}
            \hat{\mathcal{G}}_{H}^{-1}      & \mathbf{0} \\
            \mathbf{0}    & \bmat{A}_{FF}^{-1}
        \end{array} \right] 
        \left[ \begin{array}{c c}
            \bmat{S}      & \mathbf{0} \\
            \mathbf{0}    & \bmat{A}_{FF}
        \end{array} \right] 
        \left[ \begin{array}{c c}
            \bmat{I}              & \mathbf{0} \\
            \bmat{A}_{FF}^{-1} \bmat{A}_{FC}    & \bmat{I}
        \end{array} \right] \\
         = &
        \left[ \begin{array}{c c}
            \bmat{I}              & \mathbf{0} \\
            \bmat{A}_{FF}^{-1} \bmat{A}_{FC}    & \bmat{I}
        \end{array} \right]^{-1}
        \left[ \begin{array}{c c}
            \hat{\mathcal{G}}_{H}^{-1} \bmat{S}     & \mathbf{0} \\
            \mathbf{0}    & \bmat{I}
        \end{array} \right] 
        \left[ \begin{array}{c c}
            \bmat{I}              & \mathbf{0} \\
            \bmat{A}_{FF}^{-1} \bmat{A}_{FC}    & \bmat{I}
        \end{array} \right],
    \end{align*}
    Therefore, matrix $\mathcal{M}^{-1} \mathcal{A}$ shares the same eigenvalues as
    $$
        \left[ \begin{array}{c c}
            \hat{\mathcal{G}}_{H}^{-1} \bmat{S}     & \mathbf{0} \\
            \mathbf{0}    & \bmat{I} 
        \end{array} \right],
    $$
    proving Equation~\eqref{eqn:M-F-AMG-eigenvalue}.
\end{proof}

Under the assumptions of Theorem~\ref{theorem:M-F-AMG}, the eigenvalues of $\mathcal{M}^{-1} \mathcal{A}$ depend on how well $\hat{\mathcal{G}}_{H}$ approximates $\bmat{S}$. The higher the accuracy of AMG in solving the coarse grid system $\hat{\mathcal{A}}_{H} \bvec{e}_{H} = \bvec{f}_{H}$, the better the convergence of the two-level method.

\subsubsection{sSIMPLE smoother}

In this subsection, we propose a simplified SIMPLE method (sSIMPLE) and analyze its effectiveness as a smoother in the context of the two-level method. Based on the block form of the matrix $\mathcal{A}$ in Equation~\eqref{eqn:CF}, if $\bmat{A}_{CC}$ is invertible, then
\begin{align}\label{eqn:A-LU}
    \mathcal{A} = \left[\begin{array}{cc}
        \bmat{A}_{CC} & \mathbf{0}  \\
        \bmat{A}_{FC} & \bmat{S}_{(C)} 
    \end{array}\right]
    \left[\begin{array}{cc}
        \bmat{I}   & \bmat{A}_{CC}^{-1}\bmat{A}_{CF}  \\
        \mathbf{0} & \bmat{I} 
    \end{array}\right],
\end{align}
where $\bmat{S}_{(C)} = \bmat{A}_{FF} - \bmat{A}_{FC}\bmat{A}_{CC}^{-1}\bmat{A}_{CF}$ is the Schur complement of $\bmat{A}_{CC}$ in $\mathcal{A}$. Based on this decomposition, a preconditioner can be constructed that involves the computation of $\bmat{A}_{CC}^{-1}$. This computation is prohibitively expensive in practical applications. To overcome this issue, we define the following preconditioner:
\begin{align}\label{eqn:M-LU}
    \mathcal{B} = \left[\begin{array}{cc}
        \tilde{\bmat{A}}_{CC} & \mathbf{0}  \\
        \bmat{A}_{FC} & \tilde{\bmat{S}}_{(C)} 
    \end{array}\right]
    \left[\begin{array}{cc}
        \bmat{I}   & \tilde{\bmat{A}}_{CC}^{-1}\bmat{A}_{CF}  \\
        \mathbf{0} & \bmat{I} 
    \end{array}\right],
\end{align}
where $\tilde{\bmat{S}}_{(C)} = \bmat{A}_{FF} - \bmat{A}_{FC}\tilde{\bmat{A}}_{CC}^{-1}\bmat{A}_{CF}$, and $\tilde{\bmat{A}}_{CC}$ is an approximation of $\bmat{A}_{CC}$. In Equation~\eqref{eqn:M-LU}, there are three instances where $\tilde{\bmat{A}}_{CC}$ is involved: the main diagonal block, the approximate Schur complement, and the upper right block. 

If a diagonal approximation is used, i.e., $\tilde{\bmat{A}}_{CC} = \text{diag}\left( \bmat{A}_{CC}\right) := \bmat{D}_{CC}$, then $\tilde{\bmat{S}}_{(C)} = \bmat{A}_{FF} - \bmat{A}_{FC}\bmat{D}_{CC}^{-1}\bmat{A}_{CF}$. We refer to this method as the simplified SIMPLE method (sSIMPLE). The specific form of the sSIMPLE preconditioner is given by:
\begin{align}\label{eqn:M-sSIMPLE}
    \mathcal{B}_{s} = \left[\begin{array}{cc}
        \bmat{D}_{CC} & \mathbf{0}  \\
        \bmat{A}_{FC} & \tilde{\bmat{S}}_{(C)} 
    \end{array}\right]
    \left[\begin{array}{cc}
        \bmat{I}   & \bmat{D}_{CC}^{-1}\bmat{A}_{CF}  \\
        \mathbf{0} & \bmat{I} 
    \end{array}\right]
    =
    \left[\begin{array}{cc}
        \bmat{D}_{CC}   & \bmat{A}_{CF}  \\
        \bmat{A}_{FC}   & \bmat{A}_{FF} 
    \end{array}\right].
\end{align}

When using the sSIMPLE method for smoothing, denote the smoothed vector as $\bvec{x}_{s} = \mathcal{B}_{s}^{-1} \bvec{b}$, i.e., solving $\mathcal{B}_{s} \bvec{x}_{s} = \bvec{b}$. The iterative process can be carried out in two steps:
\begin{enumerate}
    \item[(1)] Solve
    \begin{align}
        \left[\begin{array}{cc}
            \bmat{D}_{CC} & \mathbf{0}  \\
            \bmat{A}_{FC} & \tilde{\bmat{S}}_{(C)} 
        \end{array}\right]
        \left[\begin{array}{c}
            \bvec{p}  \\
            \bvec{q} 
        \end{array}\right]
        =
        \left[\begin{array}{c}
            \bvec{b}_{C}  \\
            \bvec{b}_{F} 
        \end{array}\right],
    \end{align}
    which involves solving two subsystems: first, compute $\bvec{p} = \bmat{D}_{CC}^{-1}  \bvec{b}_{C}$, and then solve $\tilde{\bmat{S}}_{(C)} \bvec{q} = \bvec{b}_{F} - \bmat{A}_{FC} \bvec{p}$. 
    \item[(2)] Back-substitution
    \begin{align}
        \left[\begin{array}{cc}
            \bmat{I}   & \bmat{D}_{CC}^{-1}\bmat{A}_{CF}  \\
            \mathbf{0} & \bmat{I} 
        \end{array}\right]
        \left[ \begin{array}{c}
            \bvec{x}_{s,C} \\
            \bvec{x}_{s,F} 
        \end{array}\right]
        =
        \left[\begin{array}{c}
            \bvec{p}  \\
            \bvec{q} 
        \end{array}\right],        
    \end{align}
    setting $\bvec{x}_{s,F} = \bvec{q}$, and then compute $\bvec{x}_{s,C} = \bvec{p} - \bmat{D}_{CC}^{-1}\bmat{A}_{CF} \bvec{q}$ to obtain the solution vector $\bvec{x}_{s}$.
\end{enumerate}
Thus, the smoothed vector can be expressed as:
\begin{align} \label{eqn:x-sSIMPLE}
    \bvec{x}_{s} = 
    \left[ \begin{array}{c}
        \bvec{x}_{s,C} \\
        \bvec{x}_{s,F} 
    \end{array}\right] =
    \left[ \begin{array}{c}
        \bmat{D}_{CC}^{-1} \bvec{b}_{C} - \bmat{D}_{CC}^{-1} \bmat{A}_{CF} \bvec{q} \\
        \bvec{q} 
    \end{array}\right].
\end{align}
As can be seen from the above process, the single-step iteration of the sSIMPLE method only requires solving a linear system involving the Schur complement, while the rest involves matrix-vector multiplications and vector corrections. Therefore, the computational complexity of this method is relatively low, making it suitable as a smoother for the two-level method.

From Proposition~\ref{prop:strategy-one}, we know that when the coarse grid system is solved exactly, the residual vector after one iteration of the two-level algorithm can be expressed in terms of the smoothed vector. Substituting the smoothed vector obtained from sSIMPLE in Equation~\eqref{eqn:x-sSIMPLE} into the residual expression in Equation~\eqref{eqn:residual}, we have:
\begin{align*}
    \bvec{r}  
    &=
    \left[ \begin{array}{c}
        \bmat{A}_{CF} \bmat{A}_{FF}^{-1} \bvec{b}_{F} - \bmat{A}_{CF} \bvec{x}_{s,F} - \bmat{A}_{CF} \bmat{A}_{FF}^{-1} \bmat{A}_{FC} \bvec{x}_{s,C}\\
        \bvec{b}_{F} - \bmat{A}_{FF} \bvec{x}_{s,F} - \bmat{A}_{FC} \bvec{x}_{s,C}  
    \end{array}\right]  \\
    &=
    \left[ \begin{array}{c}
        \bmat{A}_{CF} \bmat{A}_{FF}^{-1} (\bvec{b}_{F} - \bmat{A}_{FC} \bmat{D}_{CC}^{-1} \bvec{b}_{C}) - \bmat{A}_{CF} \bmat{A}_{FF}^{-1} \tilde{\bmat{S}}_{(C)} \bvec{q}\\
        \bvec{b}_{F} - \bmat{A}_{FC} \bmat{D}_{CC}^{-1} \bvec{b}_{C} - \tilde{\bmat{S}}_{(C)} \bvec{q}
    \end{array}\right]     \\
    &=
    \left[ \begin{array}{c}
        \bmat{A}_{CF} \bmat{A}_{FF}^{-1} \bvec{r}_{\tilde{S}}\\
        \bvec{r}_{\tilde{S}}
    \end{array}\right] ,
\end{align*}
where $\bvec{r}_{\tilde{S}} = \bvec{b}_{F} - \bmat{A}_{FC} \bmat{D}_{CC}^{-1} \bvec{b}_{C} - \tilde{\bmat{S}}_{(C)} \bvec{q}$ represents the residual of the approximate Schur complement subsystem $\tilde{\bmat{S}}_{(C)} \bvec{q} =  \bvec{b}_{F} - \bmat{A}_{FC} \bvec{q}$. 

Further estimating the 1-norm of the residual $\bvec{r}$, we obtain:
\begin{align}
    \| \bvec{r} \|_{1} &= \| \bmat{A}_{CF} \bmat{A}_{FF}^{-1} \bvec{r}_{\tilde{S}} \|_{1} + \| \bvec{r}_{\tilde{S}} \|_{1} \\
    &\leq \left(1 + \| \bmat{A}_{CF} \bmat{A}_{FF}^{-1} \|_{1} \right)  \| \bvec{r}_{\tilde{S}} \|_{1},
\end{align}
Thus, when the subsystem is solved to a certain accuracy, i.e., $\| \bvec{r}_{\tilde{S}} \| = \| \bvec{b}_{F} - \bmat{A}_{FC} \bmat{D}_{CC}^{-1} \bvec{b}_{C} - \tilde{\bmat{S}}_{(C)} \bvec{q} \| < \varepsilon$, the residual of the original system satisfies
\begin{eqnarray*}
    \| \bvec{r} \| \leq C \varepsilon, \quad C = 1 + \| \bmat{A}_{CF} \bmat{A}_{FF}^{-1} \|,
\end{eqnarray*}
leading to the following proposition:
\begin{proposition}\label{prop:M-sSIMPLE-residual}
    With the restriction operator $\hat{\mathcal{R}}$ and interpolation operator $\hat{\mathcal{P}}$ defined by Equations~\eqref{eqn:ideal-R} and \eqref{eqn:ideal-P}, and the coarse grid operator defined by Equation~\eqref{eqn:ideal-coarse} as $\bmat{\mathcal{A}}_{H} = \hat{\mathcal{R}} \mathcal{A} \hat{\mathcal{P}}$, using Algorithm~\ref{algorithm:solve-strategy-one} to solve the linear system $\mathcal{A} \bvec{x} = \bvec{b}$, where the coarse grid system $\bmat{\mathcal{A}_{H}} \bvec{e}_{H} = \bvec{f}_{H}$ is solved exactly, if sSIMPLE is used as the smoother, as defined in Equation~\eqref{eqn:M-sSIMPLE}, and $\| \bvec{b}_{F} - \bmat{A}_{FC} \bmat{D}_{CC}^{-1} \bvec{b}_{C} - \tilde{\bmat{S}}_{(C)} \bvec{q} \| < \varepsilon$, then after one iteration of the two-level algorithm, the residual of the original system satisfies
    \begin{align}
         \| \bvec{r} \| \leq  C \varepsilon, \quad C = 1 + \| \bmat{A}_{CF} \bmat{A}_{FF}^{-1} \|.
    \end{align}
\end{proposition}

From Proposition~\ref{prop:M-sSIMPLE-residual}, it is evident that sSIMPLE smoothing can match the ideal interpolation operator and serves as an effective smoother. The higher the accuracy in solving the subsystem involving $\tilde{\bmat{S}}_{(C)}$, the faster the convergence of the two-level method. In practice, methods such as incomplete LU decomposition can be used to solve the subsystem $\tilde{\bmat{S}}_{(C)} \bvec{z} = \bvec{r}$.

In summary, for the case of the ideal restriction operator $\hat{\mathcal{R}}$ and ideal interpolation operator $\hat{\mathcal{P}}$ defined by Equations~\eqref{eqn:ideal-R} and \eqref{eqn:ideal-P}, we have the following conclusions:
\begin{enumerate}
    \item[(1)] When smoothing is performed only on the fine points, and $\| \bvec{b}_{F} - \bmat{A}_{FF} \bvec{x}_{s,F} \| < \varepsilon$:
    \begin{enumerate}
        \item[~a.~] If the coarse grid system is solved exactly, then after one iteration, the residual of the original system satisfies
        $$\| \bvec{r} \| \leq  C \varepsilon, \quad C = 1 + \| \bmat{A}_{CF} \bmat{A}_{FF}^{-1} \|,$$ 
        as shown in Proposition~\ref{prop:F-relaxtion-residual}.
    \end{enumerate}
    \item[(2)] When exact smoothing is performed only on the fine points, with the smoother $\mathcal{B}_{F}$ defined as in Equation~\eqref{eqn:F-relaxtion}:
    \begin{enumerate}
        \item[~a.~] If the coarse grid system is solved exactly, then the two-level method becomes a direct method, as shown in Proposition~\ref{prop:F-relaxtion-direct}.
        \item[~b.~] If AMG is used to solve the coarse grid system, then the eigenvalue distribution of the two-level method $\sigma \left(\mathcal{M}^{-1} \mathcal{A} \right)$ is as described in Theorem~\ref{theorem:M-F-AMG}.
    \end{enumerate}
    \item[(3)] When the sSIMPLE method is used for pre-smoothing, and $\| \bvec{b}_{F} - \bmat{A}_{FC} \bmat{D}_{CC}^{-1} \bvec{b}_{C} - \tilde{\bmat{S}}_{(C)} \bvec{q}  \| < \varepsilon$:
    \begin{enumerate}
        \item[~a.~] If the coarse grid system is solved exactly, then after one iteration, the residual of the original system satisfies $\| \bvec{r} \| \leq  C \varepsilon, \quad C = 1 + \| \bmat{A}_{CF} \bmat{A}_{FF}^{-1} \|$, as shown in Proposition~\ref{prop:M-sSIMPLE-residual}.
    \end{enumerate}
\end{enumerate}

\subsection{Simplified interpolation and approximate interpolation}
\label{subsect:simplified-and-approx}

\subsubsection{Simplified interpolation}
In the previous sections, we have discussed the case of the ideal interpolation operator $\hat{\mathcal{P}}$ as a component of the two-level method. However, the form of the ideal interpolation operator is relatively complex and computationally expensive. Therefore, to improve computational efficiency, it is necessary to simplify it. Based on Equation~\eqref{eqn:ideal-P}, we define the following operator:
\begin{align}\label{eqn:simple-P}
    \tilde{\mathcal{P}} = 
    \left[ \begin{array}{c}
        \bmat{I}_{CC} \\
        \tilde{\bmat{P}}_{FC} 
    \end{array}\right]
    =
    \left[ \begin{array}{c c}
        \bmat{I}_{\mathcal{N} \mathcal{N}} & \mathbf{0} \\
        \mathbf{0} & \bmat{I}_{\mathcal{M} \mathcal{M}} \\
        \mathbf{0} & \bmat{P} \\
        \mathbf{0} & \mathbf{0} 
    \end{array} \right],
\end{align}
which we refer to as the simplified interpolation operator. In essence, the simplified interpolation operator $\tilde{\mathcal{P}}$ is derived by setting the last row block of the ideal interpolation operator $\hat{\mathcal{P}}$ to zero. The corresponding coarse grid operator is then given by:
\begin{align}
    \mathcal{A}_{H} = \hat{\mathcal{R}} \mathcal{A} \tilde{\mathcal{P}}.
\end{align}
Notably, since $\hat{\mathcal{R}} \mathcal{A} \tilde{\mathcal{P}} = \bmat{S} = \hat{\mathcal{R}} \mathcal{A} \hat{\mathcal{P}}$, the coarse grid operator constructed using the simplified interpolation operator is the same as the one constructed using the ideal interpolation operator, both being the Schur complement matrix.

In the following lemmas and theorems, we will discuss the simplified interpolation operator to reveal its effectiveness in the two-level method.

\begin{lemmaNoParens}[~\cite{wiesner2014multigrid}]\label{lemma:A-E-block}
    Let the interpolation and restriction operators be defined as follows:
    \begin{align}
    \mathcal{P} = \left[ \begin{array}{c}
        \bmat{I}_{CC} \\
        \bmat{P}_{FC} 
    \end{array}\right],
    \quad
    \mathcal{R} = \left[ \begin{array}{cc}
        \bmat{I}_{CC} & \bmat{R}_{CF} 
    \end{array}\right],
    \end{align}
    then the matrix $\mathcal{A}$ can be represented as
    \begin{align}\label{eqn:A-E-LDU}
        \mathcal{A} = 
        \left[ \begin{array}{c c}
            \bmat{A}_{CC} & \bmat{A}_{CF} \\
            \bmat{A}_{FC} & \bmat{A}_{FF}
        \end{array} \right] 
        =
        \left[ \begin{array}{c c}
            \bmat{I}   & -\bmat{R}_{CF} \\
            \mathbf{0} & \bmat{I}
        \end{array} \right] 
        \left[ \begin{array}{c c}
            \bmat{A}_{H}  & \bmat{E}_{CF} \\
            \bmat{E}_{FC} & \bmat{A}_{FF}
        \end{array} \right] 
        \left[ \begin{array}{c c}
            \bmat{I}        & \mathbf{0} \\
            -\bmat{P}_{FC}  & \bmat{I}
        \end{array} \right],      
    \end{align}
    where $\bmat{A}_{H} = \mathcal{R} \mathcal{A} \mathcal{P}$, $\bmat{E}_{FC} = \left[ \bmat{A}_{FC} \quad \bmat{A}_{FF}\right] \mathcal{P}$, and $\bmat{E}_{CF} = \mathcal{R} \left[ \begin{array}{c}
        \bmat{A}_{CF} \\
        \bmat{A}_{FF}
    \end{array}\right].$
\end{lemmaNoParens}
\begin{proof}
    By the definitions of the matrices, $\bmat{A}_{H}, \bmat{E}_{FC}$, and $\bmat{E}_{CF}$ are given by:
    \begin{align}
        \bmat{A}_{H} &= \mathcal{R} \mathcal{A} \mathcal{P} = 
        \bmat{A}_{CC} + \bmat{R}_{CF}\bmat{A}_{FC} + \bmat{A}_{CF}\bmat{P}_{FC} +
        \bmat{R}_{CF}\bmat{A}_{FF}\bmat{P}_{FC}, \label{eqn:E-AH}\\
        \bmat{E}_{FC} &= \left[ \bmat{A}_{FC} \quad \bmat{A}_{FF}\right] \mathcal{P} =
        \bmat{A}_{FC} + \bmat{A}_{FF}\bmat{P}_{FC}, \label{eqn:E-FC}\\
        \bmat{E}_{CF} &= \mathcal{R} \left[ \begin{array}{c}
        \bmat{A}_{CF} \\
        \bmat{A}_{FF}
    \end{array}\right] = \bmat{A}_{CF} + \bmat{R}_{CF}\bmat{A}_{FF},\label{eqn:E-CF}
    \end{align}
    Substituting Equations~\eqref{eqn:E-AH}~-~\eqref{eqn:E-CF} into Equation~\eqref{eqn:A-E-LDU}, we obtain
    \begin{align*}
        \left[ \begin{array}{c c}
            \bmat{I}   & -\bmat{R}_{CF} \\
            \mathbf{0} & \bmat{I}
        \end{array} \right] 
        \left[ \begin{array}{c c}
            \bmat{A}_{H}  & \bmat{E}_{CF} \\
            \bmat{E}_{FC} & \bmat{A}_{FF}
        \end{array} \right] 
        \left[ \begin{array}{c c}
            \bmat{I}        & \mathbf{0} \\
            -\bmat{P}_{FC}  & \bmat{I}
        \end{array} \right] 
        =
        \left[ \begin{array}{c c}
            \bmat{A}_{CC} & \bmat{A}_{CF} \\
            \bmat{A}_{FC} & \bmat{A}_{FF}
        \end{array} \right].      
    \end{align*}
\end{proof}

This lemma demonstrates that even if the interpolation and restriction operators are not ideal, by introducing the error matrices $\bmat{E}_{FC}$ and $\bmat{E}_{CF}$, we can still represent the matrix $\mathcal{A}$ in block form.

\begin{theorem}\label{theorem:tilde-P-AMG}
    Let the interpolation operator $\tilde{\mathcal{P}}$ and the restriction operator $\hat{\mathcal{R}}$ be defined as in Equations~\eqref{eqn:simple-P} and~\eqref{eqn:ideal-R}, respectively:
    $$\tilde{\mathcal{P}} = 
    \left[ \begin{array}{c}
        \bmat{I}_{CC} \\
        \tilde{\bmat{P}}_{FC} 
    \end{array}\right], \quad
    \hat{\mathcal{R}}=
    \left[ \begin{array}{c c}
        \bmat{I}_{CC}  & \hat{\bmat{R}}_{CF}  
    \end{array} \right],$$
    and define the coarse grid operator as $\mathcal{A}_{H} = \hat{\mathcal{R}} \mathcal{A} \tilde{\mathcal{P}}$. The smoother is defined as $\mathcal{B}_{F}^{-1} = \mathcal{Q} (\mathcal{Q}^{\top} \mathcal{A}\mathcal{Q})^{-1} \mathcal{Q}^{\top}$. Using AMG to solve the coarse grid system $\mathcal{A}_{H} \bvec{e}_{H} = \bvec{f}_{H}$, the associated operator is denoted as $\mathcal{G}_{H}^{-1}$. Define
    \begin{align}\label{eqn:M-simple-P-block-LDU}
        \mathcal{M} 
        =
        \left[ \begin{array}{c c}
            \bmat{I}   & -\hat{\bmat{R}}_{CF}   \\
            \mathbf{0} & \bmat{I}
        \end{array} \right] 
        \left[ \begin{array}{c c}
            \mathcal{G}_{H} & \mathbf{0} \\
            \mathbf{0}      & \bmat{A}_{FF}
        \end{array} \right] 
        \left[ \begin{array}{c c}
            \bmat{I}                & \mathbf{0} \\
            -\tilde{\bmat{P}}_{FC}  & \bmat{I}
        \end{array} \right],
    \end{align}
    then the iteration matrix of the two-level method is
    \begin{align}\label{eqn:M-simple-P-iter}
        \mathcal{I} - \mathcal{M}^{-1} \mathcal{A} 
        = \left(\mathcal{I} - \tilde{\mathcal{P}} \mathcal{G}_{H}^{-1} \hat{\mathcal{R}} \mathcal{A} \right) \left(\mathcal{I} -\mathcal{B}_{F}^{-1}  \mathcal{A} \right),
    \end{align}   
    and the eigenvalue distribution of $\mathcal{M}^{-1} \mathcal{A}$ is
    \begin{align}\label{eqn:M-simple-P-eigenvalue}
        \sigma\left( \mathcal{M}^{-1} \mathcal{A} \right) =  \left\{ 1 \right\} \cup  \left\{ \sigma_{j}: \; \sigma_{j} \in \sigma\left( \mathcal{G}_{H}^{-1} \bmat{S} \right) \right\},
    \end{align}
    where $\sigma(\bmat{K})$ denotes the spectrum of the matrix $\bmat{K}$. 
\end{theorem}
\begin{proof}
    Notice that $\hat{\mathcal{R}} \mathcal{A} \mathcal{Q} = \mathbf{0}$ holds, and by Lemma~\ref{lemma:add-equal-multi}, Equation~\eqref{eqn:M-simple-P-iter} follows. Furthermore, by Lemma~\ref{lemma:A-E-block}, the matrix $\mathcal{A}$ can be expressed as:
    \begin{align} \label{eqn:A-E-simple-P-block}
        \mathcal{A} = 
        \left[ \begin{array}{c c}
            \bmat{A}_{CC} & \bmat{A}_{CF} \\
            \bmat{A}_{FC} & \bmat{A}_{FF}
        \end{array} \right] 
        =
        \left[ \begin{array}{c c}
            \bmat{I}   & -\hat{\bmat{R}}_{CF} \\
            \mathbf{0} & \bmat{I}
        \end{array} \right] 
        \left[ \begin{array}{c c}
            \mathcal{A}_{H}  & \bmat{E}_{CF} \\
            \bmat{E}_{FC}    & \bmat{A}_{FF}
        \end{array} \right] 
        \left[ \begin{array}{c c}
            \bmat{I}        & \mathbf{0} \\
            -\tilde{\bmat{P}}_{FC}  & \bmat{I}
        \end{array} \right],
    \end{align}
    where $\mathcal{A}_{H} = \hat{\mathcal{R}} \mathcal{A} \tilde{\mathcal{P}} = \bmat{S}$, $\bmat{E}_{FC} = \bmat{A}_{FC} + \bmat{A}_{FF} \tilde{\bmat{P}}_{FC}$, and $\bmat{E}_{CF} = \bmat{A}_{CF} + \hat{\bmat{R}}_{CF} \bmat{A}_{FF} = \mathbf{0}$.
    Combining Equations~\eqref{eqn:M-simple-P-block-LDU} and~\eqref{eqn:A-E-simple-P-block}, we obtain
    \begin{align*}
        \mathcal{M}^{-1} \mathcal{A} = &
        \left[ \begin{array}{c c}
            \bmat{I}                & \mathbf{0} \\
            -\tilde{\bmat{P}}_{FC}  & \bmat{I}
        \end{array} \right]^{-1}
        \left[ \begin{array}{c c}
            \mathcal{G}_{H}^{-1}    & \mathbf{0} \\
            \mathbf{0}              & \bmat{A}_{FF}^{-1}
        \end{array} \right] 
        \left[ \begin{array}{c c}
            \bmat{S}         & \mathbf{0} \\
            \bmat{E}_{FC}    & \bmat{A}_{FF}
        \end{array} \right] 
        \left[ \begin{array}{c c}
            \bmat{I}                & \mathbf{0} \\
            -\tilde{\bmat{P}}_{FC}  & \bmat{I}
        \end{array} \right]    \\
         = &
        \left[ \begin{array}{c c}
            \bmat{I}                & \mathbf{0} \\
            -\tilde{\bmat{P}}_{FC}  & \bmat{I}
        \end{array} \right]^{-1}
        \left[ \begin{array}{c c}
            \mathcal{G}_{H}^{-1} \bmat{S}      & \mathbf{0} \\
            \bmat{A}_{FF}^{-1}\bmat{E}_{FC}    & \bmat{I} 
        \end{array} \right] 
        \left[ \begin{array}{c c}
            \bmat{I}                & \mathbf{0} \\
            -\tilde{\bmat{P}}_{FC}  & \bmat{I}
        \end{array} \right],
    \end{align*}
    meaning that the matrix $\mathcal{M}^{-1} \mathcal{A}$ is similar to
    $$
        \left[ \begin{array}{c c}
            \mathcal{G}_{H}^{-1} \bmat{S}      & \mathbf{0} \\
            \bmat{A}_{FF}^{-1}\bmat{E}_{FC}    & \bmat{I} 
        \end{array} \right],
    $$
    hence they share the same eigenvalues. The eigenvalues of the latter can be expressed as
    $$\left\{ 1 \right\} \cup  \left\{ \sigma_{j}: \; \sigma_{j} \in \sigma\left( \mathcal{G}_{H}^{-1} \bmat{S} \right) \right\}, $$
    the theorem is proved. 
\end{proof}

\subsubsection{Approximate interpolation}

We have defined two types of interpolation operators: the ideal interpolation operator $\hat{\mathcal{P}}$ and the simplified interpolation operator $\tilde{\mathcal{P}}$, as shown in Equations~\eqref{eqn:ideal-P} and~\eqref{eqn:simple-P}, respectively. The coarse grid operator, in both cases, is the Schur complement matrix, i.e.,
\begin{align} 
    {\mathcal{A}}_{H}
    =
    \left[ \begin{array}{c c}
	\bmat{K}_{\mathcal{N} \mathcal{N}} & \bmat{K}_{\mathcal{N} \mathcal{M}} + \bmat{K}_{\mathcal{N} \mathcal{S}} \bmat{P} \\
	\bmat{K}_{\mathcal{M} \mathcal{N}} + \bmat{P}^{\top} \bmat{K}_{\mathcal{S} \mathcal{N}} & \bmat{K}_{\mathcal{M} \mathcal{M}} + \bmat{P}^{\top} \bmat{K}_{\mathcal{S} \mathcal{S}} \bmat{P}
    \end{array} \right].
\end{align}
It should be noted that both the interpolation operator and the coarse grid matrix contain the submatrix $\bmat{P}$, where $\bmat{P} = \bmat{D}^{-1} \bmat{M}$. Due to the inversion of the mortar matrix $\bmat{D}^{-1}$ in the calculation of $\bmat{P}$, the matrix $\bmat{P}$ contains many non-zero elements. This results in certain subblocks of the interpolation matrix and the coarse grid operator having a high density of non-zero elements, which negatively impacts computational efficiency. To ensure sparsity, we define an approximate matrix $\bmat{P}^{d}$ as follows:
\begin{align}
    \bmat{P}^{d}_{ij} = \left\{ \begin{array}{cl}
        \bmat{P}_{ij} & \text{if} \, |\bmat{P}_{ij}| > \varepsilon \\
        0             & \text{otherwise}
    \end{array} \right.,
\end{align}
where $\varepsilon > 0$ is a small, user-defined parameter (e.g., $\varepsilon = 10^{-10}$). This means that in constructing $\bmat{P}^{d}$, we discard the elements of $\bmat{P}$ with absolute values smaller than $\varepsilon$. Based on $\bmat{P}^{d}$, we define the approximate ideal interpolation operator $\hat{\mathcal{P}}^{d}$ and the approximate simplified interpolation operator $\tilde{\mathcal{P}}^{d}$ as follows:
\begin{align}
    \hat{\mathcal{P}}^{d}
    &=
    \left[ \begin{array}{c}
        \bmat{I}_{CC} \\
        \hat{\bmat{P}}_{FC}^{d} 
    \end{array}\right]
    =
    \left[ \begin{array}{c c}
        \bmat{I}_{\mathcal{N} \mathcal{N}} & \mathbf{0} \\
        \mathbf{0} & \bmat{I}_{\mathcal{M} \mathcal{M}} \\
        \mathbf{0} & \bmat{P}^{d} \\
        -\bmat{D}^{-\top} \bmat{K}_{\mathcal{S} \mathcal{N}} & -\bmat{D}^{-\top} \bmat{K}_{\mathcal{S} \mathcal{S}} \bmat{P}^{d}
    \end{array} \right],  \label{eqn:drop-ideal-P}\\
    \tilde{\mathcal{P}}^{d}
    &=
    \left[ \begin{array}{c}
        \bmat{I}_{CC} \\
        \tilde{\bmat{P}}_{FC}^{d} 
    \end{array}\right]
    =
    \left[ \begin{array}{c c}
        \bmat{I}_{\mathcal{N} \mathcal{N}} & \mathbf{0} \\
        \mathbf{0} & \bmat{I}_{\mathcal{M} \mathcal{M}} \\
        \mathbf{0} & \bmat{P}^{d} \\
        \mathbf{0} & \mathbf{0} 
    \end{array} \right],    \label{eqn:drop-simple-P}
\end{align}
as well as the approximate coarse grid operator ${\mathcal{A}}_{H}^{d}$:
\begin{align} 
    {\mathcal{A}}_{H}^{d}
    =
    \left[ \begin{array}{c c}
	\bmat{K}_{\mathcal{N} \mathcal{N}} & \bmat{K}_{\mathcal{N} \mathcal{M}} + \bmat{K}_{\mathcal{N} \mathcal{S}} \bmat{P}^{d} \\
	\bmat{K}_{\mathcal{M} \mathcal{N}} + \left(\bmat{P}^{d}\right)^{\top} \bmat{K}_{\mathcal{S} \mathcal{N}} & \bmat{K}_{\mathcal{M} \mathcal{M}} + \left(\bmat{P}^{d}\right)^{\top} \bmat{K}_{\mathcal{S} \mathcal{S}} \bmat{P}^{d}
    \end{array} \right].
\end{align}

\subsection{Design and Implementation of two-level AMG}
\label{subsect:design}

The definition of the simplified interpolation operator $\tilde{\mathcal{P}}$ is given by Equation~\eqref{eqn:simple-P}. Its simple form makes it easy to implement in practice. Next, we will focus on the implementation of the ideal interpolation operator $\hat{\mathcal{P}}$ and the ideal restriction operator $\hat{\mathcal{R}}$.

In classical AMG methods, the ideal interpolation operator includes the block $\hat{\bmat{P}}_{FC} = -\bmat{A}_{FF}^{-1}\bmat{A}_{FC}$, which involves the inverse of the fine grid operator, leading to high computational costs. Therefore, the ideal interpolation is often not practical in many applications. However, in the two-level method for contact problems, according to the CF-splitting defined by Equation~\eqref{eqn:CF-split}, $\hat{\bmat{P}}_{FC}$ has the following form:
\begin{align}
    \hat{\bmat{P}}_{FC}
    = -\bmat{A}_{FF}^{-1}\bmat{A}_{FC}
    =
    \left[ \begin{array}{c c}
        \mathbf{0} & \bmat{P} \\
        -\bmat{D}^{-\top} \bmat{K}_{\mathcal{S} \mathcal{N}} & -\bmat{D}^{-\top} \bmat{K}_{\mathcal{S} \mathcal{S}} \bmat{P}
    \end{array} \right] .
\end{align}
In the above expression, only the inverse of the mortar matrix $\bmat{D}$ is involved. The matrix $\bmat{D}$ has favorable structural characteristics. In two-dimensional contact problems, $\bmat{D}$ is a block tridiagonal matrix, and efficient inversion methods based on the block Thomas algorithm can be used. Considering that the computation of residual restriction and error interpolation are matrix-vector multiplication operations, it is unnecessary to explicitly construct the ideal transfer operators. Instead, it suffices to represent the mathematical behavior of the matrix-vector multiplication.

\begin{itemize}
    \item[(1)] Residual Restriction Calculation $\hat{\mathcal{R}} \bvec{f}$
    
    For the residual vector $\bvec{f}$ on the fine grid, the restricted residual $\bvec{f}_{H}$ can be expressed as:
    \begin{equation} \label{eqn:restric-residual}
        \bvec{f}_{H}
        =
        \hat{\mathcal{R}} \bvec{f}
        =
        \left[ \begin{array}{c}
            \bvec{f}_{\mathcal{N}} - \bmat{K}_{\mathcal{N} \mathcal{S}}\bmat{D}^{-1} \bvec{f}_\lambda \\
            \bvec{f}_{\mathcal{M}} + \bmat{P}^{\top} \bvec{f}_{\mathcal{S}} - \bmat{P}^\top \bmat{K}_{\mathcal{S} \mathcal{S}}\bmat{D}^{-1} \bvec{f}_\lambda
        \end{array} \right].
    \end{equation}  
    Define the vector $\tilde{\bvec{v}}$
    \begin{align}
        \tilde{\bvec{v}}
        =
        \left[ \begin{array}{c}
            \bvec{0}  \\
            \bvec{0}  \\
            \bmat{D}^{-1} \bvec{f}_\lambda \\
            \bvec{0}
        \end{array} \right], 
    \end{align}
    then it follows that 
    \begin{align}
        \bmat{\mathcal{A}} \tilde{\bvec{v}}
        =
        \left[ \begin{array}{c}
            \bmat{K}_{\mathcal{N} \mathcal{S}} \bmat{D}^{-1} \bvec{f}_\lambda \\
            \bvec{0} \\
            \bmat{K}_{\mathcal{S} \mathcal{S}} \bmat{D}^{-1} \bvec{f}_\lambda \\
            \bvec{f}_\lambda
        \end{array} \right].        
    \end{align}
    During the construction of $\tilde{\bvec{v}}$, only the calculation of $\bmat{D}^{-1} \bvec{f}_\lambda$ is required. After rearrangement, the mortar matrix $\bmat{D}$ can be transformed into a symmetric block tridiagonal matrix $\tilde{\bmat{D}}$, satisfying $\bmat{D} = \tilde{\bmat{D}} \bmat{T}$, where $\bmat{T}$ is a permutation matrix. Thus, 
    \[ \bmat{D}^{-1} \bvec{f}_\lambda = \bmat{T}^{-1} (\tilde{\bmat{D}}^{-1} \bvec{f}_\lambda), \]
    where the calculation of $\tilde{\bmat{D}}^{-1} \bvec{f}_\lambda$ can be efficiently implemented using the block Thomas algorithm.
    
    The residual restriction~\eqref{eqn:restric-residual}~can be expressed as: 
    \begin{align}
        \hat{\mathcal{R}} \bvec{f} = \tilde{\mathcal{R}} \bvec{f} - \tilde{\mathcal{R}} \bmat{\mathcal{A}} \tilde{\bvec{v}} =  \tilde{\mathcal{R}} (\bvec{f} - \bmat{\mathcal{A}} \tilde{\bvec{v}}),
    \end{align}
    where $\tilde{\mathcal{R}} = \tilde{\mathcal{P}}^{\top}$ is the simplified restriction operator. Therefore, the matrix-vector multiplication of the ideal restriction operator is equivalent to the multiplication of the simplified interpolation operator with another vector.
    
    \item[(2)] Interpolation Calculation $\hat{\mathcal{P}} \bvec{e}_{H}$
    
    In strategy one, using the ideal interpolation operator $\hat{\mathcal{P}}$, the interpolated vector can be expressed as:
    \begin{align}\label{eqn:interp-error-one}
        \bvec{e}
        =
        \hat{\mathcal{P}} \bvec{e}_{H}
        =
        \left[ \begin{array}{c}
            \bvec{e}_{H,\mathcal{N}} \\ 
            \bvec{e}_{H,\mathcal{M}} \\
            \bmat{P} \bvec{e}_{H,\mathcal{M}} \\
            -\bmat{D}^{-\top} \bmat{K}_{\mathcal{S} \mathcal{N}} \bvec{e}_{H,\mathcal{N}} - \bmat{D}^{-\top} \bmat{K}_{\mathcal{S} \mathcal{S}} \bmat{P} \bvec{e}_{H,\mathcal{M}}
        \end{array} \right]
    \end{align}
    Define the vector $\tilde{\bvec{e}}_{1} = [\mathbf{0},\; \mathbf{0},\; \mathbf{0} ,\; ( -\bmat{D}^{-\top} \bvec{t}_{1})^{\top} ]^{\top}$, where
    \begin{align}
        \bvec{t}_{1}
        =
        \bmat{K}_{\mathcal{S} \mathcal{N}} \bvec{e}_{H,\mathcal{N}} + \bmat{K}_{\mathcal{S} \mathcal{S}} \bmat{P} \bvec{e}_{H,\mathcal{M}} 
        =
        \left( \mathcal{A} \tilde{\mathcal{P}} \bvec{e}_{H} \right)_{\mathcal{S}},
    \end{align}
    $\tilde{\mathcal{P}}$ is the simplified interpolation operator, and $(\cdot)_{\mathcal{S}}$ indicates the component corresponding to surface nodes.
    Note that the calculation involves $\bmat{D}^{-\top} \bvec{t}_{1}$. Similarly, based on $\bmat{D} = \tilde{\bmat{D}} \bmat{T}$, we have:
    \begin{align}
        \bmat{D}^{-\top}  = (\tilde{\bmat{D}} \bmat{T})^{-\top} = (\bmat{T}^{\top} \tilde{\bmat{D}}^{\top})^{-1} = (\bmat{T}^{-1} \tilde{\bmat{D}})^{-1} = \tilde{\bmat{D}}^{-1} \bmat{T},
    \end{align}
    Therefore, $\bmat{D}^{-\top} \bvec{t}_{1} = \tilde{\bmat{D}}^{-1} \bmat{T} \bvec{t}_{1} = \tilde{\bmat{D}}^{-1} \tilde{\bvec{t}}_{1}$, meaning that $\bmat{D}^{-\top} \bvec{t}_{1}$ can still be calculated using the block Thomas algorithm.

    The above analysis shows that the interpolation calculation $\hat{\mathcal{P}} \bvec{e}_{H}$ can be expressed as:
    \begin{align}
        \hat{\mathcal{P}} \bvec{e}_{H} = \tilde{\mathcal{P}} \bvec{e}_{H} + \tilde{\bvec{e}}_{1} ,
    \end{align}
    In other words, the matrix-vector multiplication of the ideal interpolation operator can be expressed as the matrix-vector multiplication of the simplified interpolation operator plus a vector addition operation.
\end{itemize}

When introducing the smoother operator, we defined the exact smoother for F-points $\mathcal{B}_{F}^{-1}$, which includes the computation of $\bmat{A}_{FF}^{-1}$. Similarly, we discuss the equivalent representation of exact smoothing for F-points.

\begin{itemize}
\item[(3)] Exact Smoothing for F-points $\mathcal{B}_{F}^{-1} \bvec{b}$

In the V-cycle of the two-level method, smoothing is first performed on the fine grid level, resulting in the smoothed vector:
\begin{eqnarray*}
    {\bvec{x}}_{s} = \mathcal{B}_{F}^{-1} \bvec{b} 
    = \bmat{\mathcal{Q}} (\bmat{\mathcal{Q}}^{\top} \bmat{\mathcal{A}} \bmat{\mathcal{Q}} )^{-1} \bmat{\mathcal{Q}}^{\top} \bvec{b} 
    =
    \left[ \begin{array}{c}
        \bvec{0} \\
        \bvec{0} \\
        \bmat{D}^{-1} \bvec{b}_\lambda \\
        \bmat{D}^{-\top} \bvec{b}_{\mathcal{S}} - \bmat{D}^{-\top} \bmat{K}_{\mathcal{S} \mathcal{S}} \bmat{D}^{-1} \bvec{b}_\lambda
    \end{array} \right].
\end{eqnarray*}
Define the vectors $\hat{\bvec{v}}$ and $\hat{\bvec{e}}$ as follows:
\begin{equation*}
    \hat{\bvec{v}}
    =
    \left[ \begin{array}{c}
         \bvec{0}  \\
         \bvec{0}  \\
         \bmat{D}^{-1} \bvec{b}_\lambda \\
         \bvec{0}
    \end{array} \right], 
    \quad
    \hat{\bvec{e}}
    =
    \left[ \begin{array}{c}
        \bvec{0} \\
        \bvec{0} \\
        \bvec{0} \\
        \bmat{D}^{-\top} ( \bvec{b} - \bmat{\mathcal{A}} \hat{\bvec{v}})_{\mathcal{S}}
    \end{array} \right],
\end{equation*} 
The construction of these two vectors only involves matrix-vector multiplications with $\bmat{D}^{-1}$ and $\bmat{D}^{-\top}$. Hence, exact smoothing for F-points can be represented as:
$$ \mathcal{B}_{F}^{-1} \bvec{b} = \hat{\bvec{v}} + \hat{\bvec{e}},$$
which simplifies the implementation process in programming.

\end{itemize}

When using the two-level method as a preconditioner for Krylov subspace methods, the mathematical behavior involves solving the preconditioning equation $\mathcal{M}\bvec{z} = \bvec{r}$, where $\mathcal{M}$ is the preconditioning matrix. In this work, we set $\mathcal{M} = \mathcal{A}$ and perform only one V-cycle. The two-level preconditioning method is detailed in Algorithm~\ref{algorithm:two-level-prec}.

\begin{algorithm}
\caption{Two-Level Preconditioning Algorithm for Solving $\bvec{z} = \bmat{\mathcal{M}}{\text{TLAMG}}^{-1} \bvec{r}$} 
\label{algorithm:two-level-prec}
\begin{algorithmic}[1]
\State Function ; TLAMG($\mathcal{A}, \bvec{r}$)
\State \quad $\bvec{z} \leftarrow \bmat{\mathcal{B}}_{1}^{-1} \bvec{r}$
\Comment{Pre-smoothing}
\State \quad $\bvec{f}_{H} = \mathcal{R} \bvec{f},  \bvec{f} = \left( \bvec{r} - \mathcal{A}\bvec{z} \right)$
\Comment{Residual restriction}
\State \quad $\mathcal{A}_{H_\text{AMG}} \bvec{e}_{H} = \bvec{f}_{H}$
\Comment{Solve coarse grid system using AMG}
\State \quad $\bvec{e} = \mathcal{P} \bvec{e}_{H}$
\Comment{Interpolation}
\State \quad $\bvec{z} \leftarrow \bvec{z} + \bvec{e}$
\Comment{Error correction}
\State \quad $\bvec{z} \leftarrow \bvec{z} + \bmat{\mathcal{B}}_{2}^{-1} \left( \bvec{r} - \mathcal{A}\bvec{z} \right)$
\Comment{Post-smoothing}
\State \quad Return $\bvec{z}$
\State End Function
\end{algorithmic}
\end{algorithm}

This algorithm outlines the steps for implementing the two-level preconditioning method, which includes pre-smoothing, residual restriction, solving the coarse grid system using AMG, interpolation, error correction, and post-smoothing. By following this approach, the preconditioner effectively improves the convergence of the Krylov subspace methods.

\section{Numerical Results}
\label{sect:numer}

In this section, numerical results of the two-level preconditioning method
is given for three contact models. 
We first validate the effectiveness of two-level methods and analyze them from the perspective of algorithmic components, focusing on various smoothers, interpolation, and restriction operators.
Next, we compare the two-level methods with other preconditioning methods. Finally, we show the algorithmic scalability and parallel scalability of the two-level methods.

The numerical experiments are conducted on a cluster, 
which has 28 cores and 96G memory on each node.

\subsection{Test models}
\label{subsect:model}

Three tied contact two-dimensional models are tested.  
The specific test models are given in Fig.~\ref{model},
where the first two models have three contact bodies,
and the last one has two contact bodies.
All the test models are linear elastic, 
and the material parameters include Young's modulus $E$ and Poisson's ratio $\nu$. 
In the tests, we set $E = 20$ N/m$^2$ and $\nu = 0.3$.

For the first and second test models, 
the middle body is the master body, and the rest are slave bodies, and this corresponds to two contact surfaces.
For the third model, the body on the top is 
the master body and on the bottom is the slave body,
and one contact surface is set.

For the first model, the boundary on the left is fixed 
and on the right boundary, 
a constant surface force of $p=10 \text{N}$ in the rightward direction is set.
For the second model, the lower boundaries of the three bodies are fixed, 
and an uniform surface force of $p=10 \text{N}$ in the downward direction
is set on the upper boundaries. 
For the third model, the lower boundary of the slave body is fixed, 
and an uniformly distributed surface force of $p=-1 \text{N}$ in the downward direction
is set on the upper boundary of the master body. 
For all three models, no body force is set.

All the test models are partitioned into triangular meshes, 
and the package FreeFEM++~\cite{FREEFEMhecht2012new} is employed to discretize the tied contact problem.

\begin{figure}[H]
	\centering
	\subfigure[Model 1]{\includegraphics[scale=0.3]{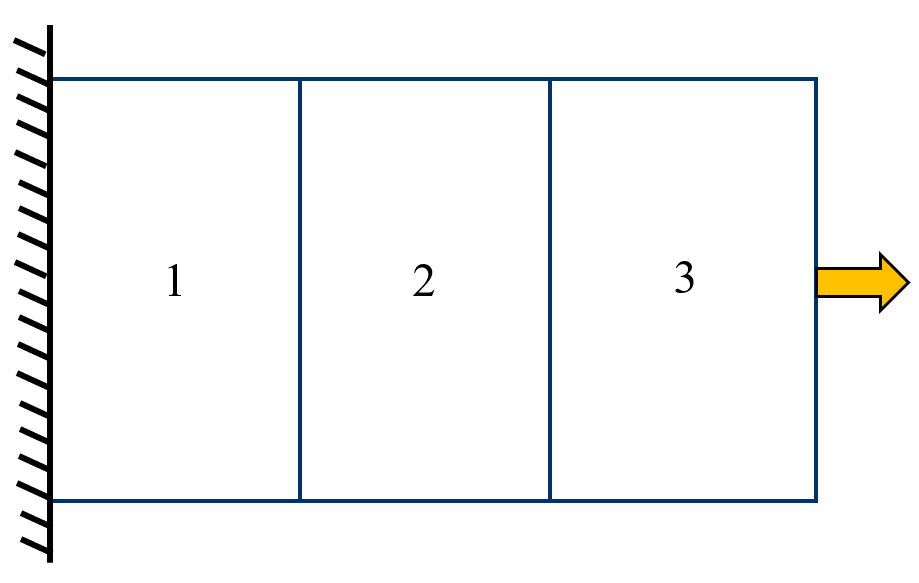} \label{model1}} 
        \subfigure[Model 2]{\includegraphics[scale=0.3]{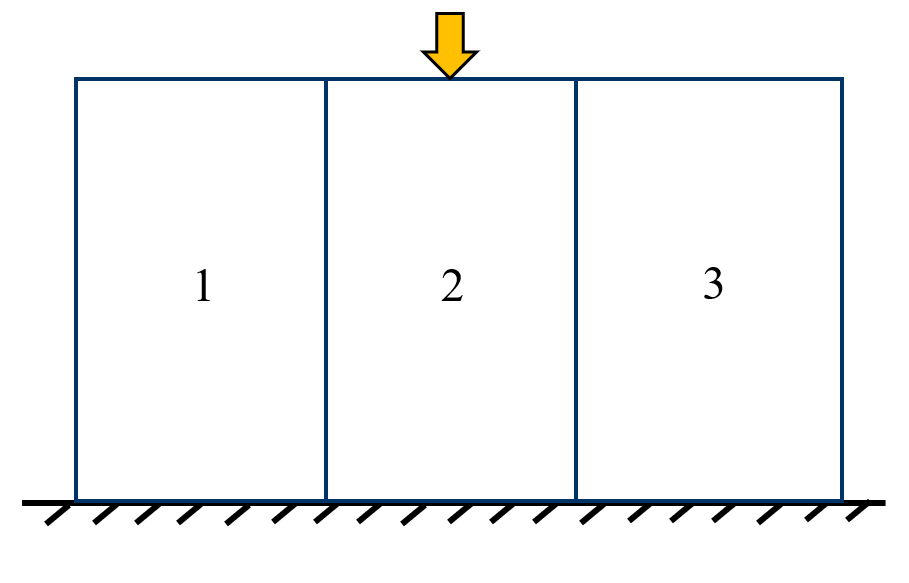} \label{model2}} 
	\subfigure[Model 3]{\includegraphics[scale=0.3]{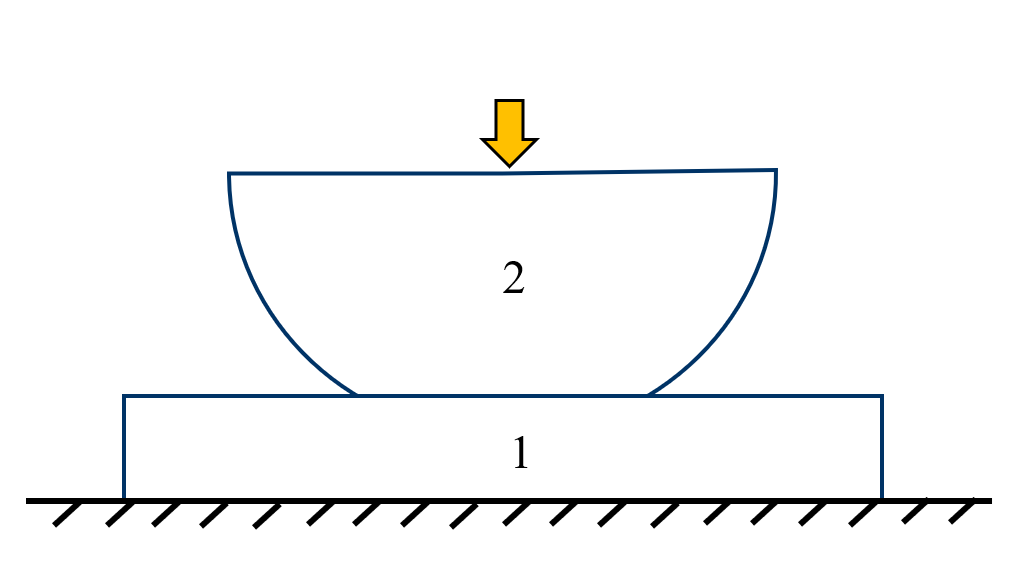} \label{model3}} 
	\caption{Tied contact example models.}
	\label{model}
\end{figure}

\subsection{Experimental setup}
\label{subsect:alg-setup}

The two-level method is implemented
by using the matrix and vector data structures provided by PETSc~\cite{PETSCbalay2017petsc}.
Besides, the linear systems are solved by calling iterative methods and preconditioners provided by PETSc.
PETSc offers several Krylov subspace iterative methods, including CG, GMRES, GCR, and MINRES, 
as well as some preconditioning methods like domain decomposition, 
algebraic multigrid, etc.
Additionally, some third party solver libraries, such as HYPRE~\cite{HYPREfalgout2002hypre},  
MUMPS~\cite{MUMPSamestoy2000mumps} can be called as preconditioning solvers in PETSc.  

For solving saddle point systems arising from contact problems, the GCR method serves as the iterative solver. Following Algorithm~\ref{algorithm:two-level-prec}, we implement a two-level preconditioning method without post-smoothing. In the design of two-level method, the transfer operators (including interpolation operator $\mathcal{P}$ and restriction operator $\mathcal{R}$) are denoted as follows:
\begin{itemize}
    \item $\hat{\mathcal{P}}$: Ideal interpolation;
    \item $\tilde{\mathcal{P}}$: Simplified interpolation;
    \item $\hat{\mathcal{R}}$: Ideal restriction;
    \item $\tilde{\mathcal{R}}$: Simplified restriction;
    \item $(\cdot)^{d}$: Approximate form, e.g., $\hat{\mathcal{P}}^{\text{d}}$ denotes approximate ideal interpolation, with similar notation for others.
\end{itemize}
The smoothing operators $\mathcal{B}$ are denoted as:
\begin{itemize}
    \item $\mathcal{B}_{F}$: Exact fine-point method;
    \item $\mathcal{B}_{s}$: Simplified SIMPLE method (sSIMPLE method);
    \item JAC: Jacobi method.
\end{itemize}
The coarse-grid system is SPD and is solved using the AMG method. Within AMG, a node coarsening strategy based on the matrix generated by the row sum norm of the eliminated matrix is employed~\cite{AMGbaker2016scalability}, and only one V-cycle is performed.

In numerical experiments, the maximum number of iterations is 100, with the convergence criterion defined as the relative residual 2-norm less than $10^{-8}$.  The time for solving the saddle point system is denoted as $\TIME$, measured in seconds (s) unless otherwise specified.  $\NIT$ denotes the number of Krylov iterations, 
and $r_{\text{rel}}^k$ represents the relative residual, that is $r_{\text{rel}}^k = \|\bvec{r}_k\|_2 / |\bvec{r}_0\|_2$.


\subsection{Results of two-level preconditioning}
\label{subsect:xxx}

This subsection first validates the effectiveness of two grid methods by testing three types of contact models. The problem sizes for Model 1 and Model 2 are 639,382, and for Model 3, it is 712,322. All tests were conducted using 28 processes. Table~\ref{tab:result} presents numerical results obtained using ideal interpolation $\hat{\mathcal{P}}$ and simplified interpolation $\tilde{\mathcal{P}}$, with the restriction operators consistently being ideal restriction operator $\hat{\mathcal{R}} = \hat{\mathcal{P}}^{\top}$. Each interpolation operator is paired with three smoothers: Jacobi smoother, exact fine-point smoother $\mathcal{B}_{F}$, and sSIMPLE smoother $\mathcal{B}_{s}$. In the setup of the sSIMPLE method, the subsystem addressing coarse points in the left-upper block employs Jacobi iteration, while the subsystem addressing fine points in the right-lower block uses ILU iteration.

\begin{table}[htbp]
  \centering
  \setlength{\tabcolsep}{4.5mm}
  \caption{Results of two-level preconditioning with different interpolations}
    \begin{tabular}{cccccccc}
    \toprule
    \multicolumn{1}{c}{\multirow{2}[2]{*}{Model}} & \multicolumn{1}{c}{\multirow{2}[2]{*}{$\mathcal{B}$}} & \multicolumn{3}{c}{$\hat{\mathcal{P}}$} & \multicolumn{3}{c}{$\tilde{\mathcal{P}}$} \\
       &   & \NIT  & \TIME  & $r_{\text{rel}}^k$ & \NIT  & \TIME  & $r_{\text{rel}}^k$ \\
    \midrule
    \multirow{3}[1]{*}{model 1} &JAC   & 100   & 5.77  & 8.89$\times 10^{-3}$ & 100   & 5.66  & 8.28$\times 10^{-1}$ \\
    &$\mathcal{B}_{F}$   & 24    & 3.13  & 5.37$\times 10^{-9}$ & 24    & 3.11  & 6.06$\times 10^{-9}$ \\
    &$\mathcal{B}_{s}$ & 25    & 3.36  & 8.90$\times 10^{-9}$ & 25    & 3.27  & 8.77$\times 10^{-9}$ \\
    \midrule
    \multirow{3}[1]{*}{model 2} &JAC   & 100   & 5.80  & 7.04$\times 10^{-3}$ & 100   & 5.50  & 2.41$\times 10^{-1}$ \\
    &$\mathcal{B}_{F}$   & 24    & 3.25  & 8.94$\times 10^{-9}$ & 25    & 3.13  & 4.23$\times 10^{-9}$ \\
    &$\mathcal{B}_{s}$ & 26    & 3.37  & 6.94$\times 10^{-9}$ & 26    & 3.28  & 6.87$\times 10^{-9}$ \\
    \midrule
    \multirow{3}[1]{*}{model 3} &JAC   & 100   & 5.39  & 3.17$\times 10^{-2}$ & 100   & 5.21  & 8.43$\times 10^{-1}$ \\
    &$\mathcal{B}_{F}$   & 30    & 3.33  & 3.78$\times 10^{-9}$ & 30    & 3.16  & 4.23$\times 10^{-9}$ \\
    &$\mathcal{B}_{s}$ & 29    & 3.51  & 4.66$\times 10^{-9}$ & 29    & 3.41  & 4.47$\times 10^{-9}$ \\
    \bottomrule
    \end{tabular}%
  \label{tab:result}%
\end{table}%

From Table~\ref{tab:result}, it is evident that when using JAC smoother, neither type of interpolation operator achieves convergence for the two-level methods. However, employing $\mathcal{B}_{F}$ and $\mathcal{B}_{s}$ smoothers leads to convergence within 30 iterations. Therefore, both types of smoothers designed for interpolation operators in this study are effective. When paired with $\mathcal{B}_{F}$ and $\mathcal{B}_{s}$ smoothers, the iteration counts remain consistent between the simplified and ideal interpolation operators. Due to its simpler form, the simplified interpolation operator  $\tilde{\mathcal{P}}$ requires less computation time compared to the ideal interpolation operator  $\hat{\mathcal{P}}$.

This paper proposes an approximate method.  
The effectiveness of approximate ideal interpolation $\hat{\mathcal{P}}^{\text{d}}$ and approximate simplified interpolation $\tilde{\mathcal{P}}^{\text{d}}$ operators, with approximate simplified restriction operators $\tilde{\mathcal{R}}^{\text{d}} = ( \tilde{\mathcal{P}}^{\text{d}} )^{\top}$, is then discussed.

For clarity, the notation $\mathcal{P}(\mathcal{B})$ denotes a two-level method that combines a pre-smoother $\mathcal{B}$ with an interpolation operator $\mathcal{P}$.
Taking model 2 as an example, we illustrate the effects of interpolation before and after approximation using the $\tilde{P}(\mathcal{B}_{s})$ method. 
Table~\ref{tab:model2-diff-drop} presents the average number of nonzeros per row $\text{nnz}(\text{row})$ for the submatrix $\bmat{P}$, interpolation matrix $\tilde{\mathcal{P}}$, and coarse grid matrix $\mathcal{A}_{h}$, alongside setup time \Tsetup, solve time \Tsolve, and total time \Ttotal ~for the two-level method.

\begin{table}[H]
  \centering
  \setlength{\tabcolsep}{4mm}
  \caption{Comparison of different approximation stategies for model 2}
    \begin{tabular}{ccccccccc}
    \toprule
    \multirow{2}[1]{*}{$\varepsilon$} & \multicolumn{3}{c}{$\text{nnz} (\text{row})$} & \multirow{2}[1]{*}{\NIT} & \multirow{2}[1]{*}{$r_{\text{rel}}^k$} & \multicolumn{3}{c}{\TIME} \\
          &  $\bmat{P}$ & $\tilde{\mathcal{P}}$ & $\mathcal{A}_{H} $    &       &       & \Tsetup & \Tsolve & \Ttotal \\
    \midrule
    no drop    & 200 & 1.6 & 17 & 26    & 6.87$\times 10^{-9}$ & 0.91  & 2.36  & 3.27  \\
    $10^{-15}$ & 3   & 1.0 & 14    & 26    & 4.42$\times 10^{-9}$ & 0.38  & 2.06  & 2.44  \\
    $10^{-10}$ & 2  & 0.9 & 13   & 26    & 5.11$\times 10^{-9}$ & 0.37  & 1.48  & 1.85  \\
    \bottomrule
    \end{tabular}%
  \label{tab:model2-diff-drop}%
\end{table}%

Table~\ref{tab:model2-diff-drop} presents each row corresponding to different interpolation operators: the original simplified interpolation $\tilde{\mathcal{P}}$, approximate simplified interpolation $\tilde{\mathcal{P}}^{\text{d}}_{1}$ with a drop criterion of $|\bmat{P}_{ij}| < 10^{-15}$, and $\tilde{\mathcal{P}}^{\text{d}}_{2}$ with a drop criterion of $|\bmat{P}_{ij}| < 10^{-10}$.
After approximation, it is noted that the number of iterations remains unchanged, but the time significantly decreases. $\tilde{\mathcal{P}}^{\text{d}}_{2}$ achieves a minimum time of 1.85s, which represents a 43.42\% reduction compared to the original simplified interpolation. 
This reduction is attributed to the fact that the original submatrix $\bmat{P}$ initially had an average of 200 nonzeros per row, which reduced to single digits after dropping. And the corresponding interpolation and coarse grid matrices become sparser.

Table~\ref{tab:result-approx} presents the results of three models using approximate interpolation. Models 1 and 2 have problem sizes of 639,382, while model 3 has a problem size of 712,322. Results for $\hat{\mathcal{P}}^{\text{d}} (\mathcal{B}_{F})$, $\hat{\mathcal{P}}^{\text{d}} (\mathcal{B}_{s})$, $\tilde{\mathcal{P}}^{\text{d}} (\mathcal{B}_{F})$, and $\tilde{\mathcal{P}}^{\text{d}} (\mathcal{B}_{s})$ are summarized. 
Compared to Table~\ref{tab:result}, the two-level methods using approximate interpolation show significant reduction in time. Across all three models, rapid convergence is achieved with iterations not exceeding 30 and time less than 2.08s.

\begin{table}[H]
  \centering
  \setlength{\tabcolsep}{4.5mm}
  \caption{Results of two-level preconditioning with different approximate interpolations (approximate simplified restriction)}
    \begin{tabular}{cccccccc}
    \toprule
    \multicolumn{1}{c}{\multirow{2}[2]{*}{Model}} & \multicolumn{1}{c}{\multirow{2}[2]{*}{$\mathcal{B}$}} & \multicolumn{3}{c}{$\hat{\mathcal{P}}^{\text{d}}$} & \multicolumn{3}{c}{$\tilde{\mathcal{P}}^{\text{d}}$} \\
       &   & \NIT  & \TIME  & $r_{\text{rel}}^k$ & \NIT  & \TIME  & $r_{\text{rel}}^k$ \\
    \midrule
    \multirow{2}[1]{*}{model 1}
    &$\mathcal{B}_{F}$   & 24    & 1.65  & 9.35$\times 10^{-9}$ & 25    & 1.57  & 4.11$\times 10^{-9}$ \\
    &$\mathcal{B}_{s}$ & 26    & 1.91  & 7.36$\times 10^{-9}$ & 26    & 1.79  & 7.23$\times 10^{-9}$ \\
    \midrule
    \multirow{2}[1]{*}{model 2} 
    &$\mathcal{B}_{F}$   & 24    & 1.65  & 6.80$\times 10^{-9}$ & 24    & 1.63  & 8.31$\times 10^{-9}$ \\
    &$\mathcal{B}_{s}$ & 26    & 1.92  & 5.19$\times 10^{-9}$ & 26    & 1.85  & 5.11$\times 10^{-9}$ \\
    \midrule
    \multirow{2}[1]{*}{model 3}
    &$\mathcal{B}_{F}$   & 30    & 1.79  & 4.68$\times 10^{-9}$ & 30    & 1.79  & 5.18$\times 10^{-9}$ \\
    &$\mathcal{B}_{s}$ & 29    & 2.08  & 5.53$\times 10^{-9}$ & 29    & 2.00  & 5.31$\times 10^{-9}$ \\
    \bottomrule
    \end{tabular}%
  \label{tab:result-approx}
\end{table}%

To visually show the performance improvement of two-level algorithms after approximating interpolation operators, Figure~\ref{fig:time} compares solving times before and after approximation for models 1 and 2.
The left vertical axis of the graph displays  the solving time of the two-level methods. The red line indicates the time for original interpolation, denoted as $t(\mathcal{P})$, while the blue line indicates the time for the approximate interpolation, denoted as $ t(\mathcal{P}^{\text{d}})$. 
The right vertical axis shows the percentage performance improvement, characterized by the reduction in time, defined as
$$
\text{perf} = \frac{t(\mathcal{P}) - t(\mathcal{P}^{\text{d}})}{t(\mathcal{P})} \times 100\%.
$$
From Figure~\ref{fig:time}, it is clear  that for each method, using approximate interpolation operators leads to significant reductions in solving time, resulting in  performance improvements exceeding 40\%. Thus, approximate interpolation proves to be a highly effective method for enhancing computational efficiency.

\begin{figure}[H]
	\centering
	\subfigure[Model 1]{\includegraphics[scale=0.37]{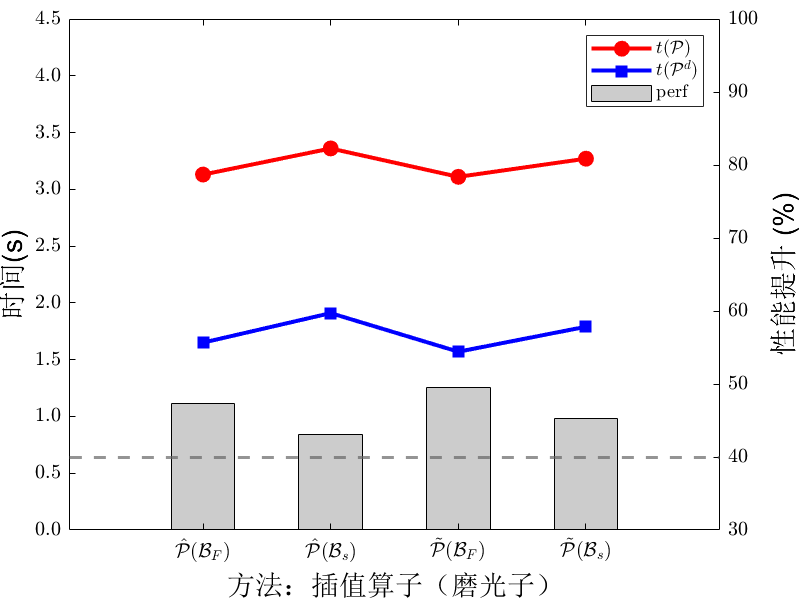} \label{fig:time-model1}} 
    \subfigure[Model 2]{\includegraphics[scale=0.37]{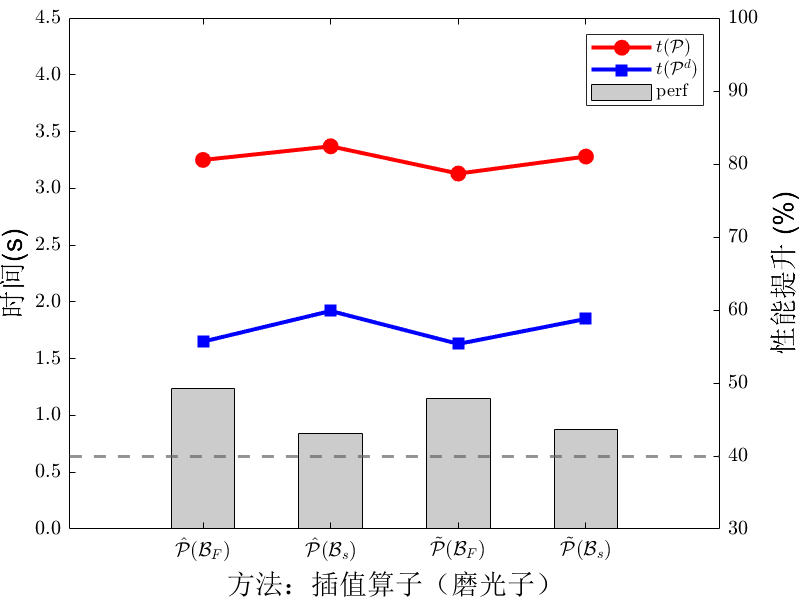} \label{fig:time-model2}} 
	\caption{Comparison of results before and after approximation.}
	\label{fig:time}
\end{figure}

\subsection{Comparison of two-level preconditioning}
\label{subsect:xxx}
To compare the two-level preconditioning methods with other preconditioning methods, the specific algorithmic components chosen for the two-level methods are $\tilde{\mathcal{P}}^{\text{d}} (\mathcal{B}_{F})$ and $\tilde{\mathcal{P}}^{\text{d}} (\mathcal{B}_{s})$, with the restriction operators being approximate simplified restriction operators $\tilde{\mathcal{R}}^{d}$. The saddle point systems are solved using GCR iteration method, and the tested preconditioning methods include:
\begin{itemize}
    \item SIMPLE: Semi-Implicit Method for Pressure Linked Equations, 
    \item AMG: Classic Algebraic Multigrid Method. 
\end{itemize}

The SIMPLE method, which was originally designed for solving block system 
in computational fluid dynamics~\cite{patankar1983calculation}, 
has also been successfully applied in contact mechanics~\cite{AMGwiesner2021algebraic}.
Thus, it is tested as one preconditioning method
for solving the saddle system.
In the implementation of the SIMPLE method,
it is needed to provide solution methods 
for solving the subsystem corresponding to 
the left-upper and right-lower blocks of the saddle point matrix, which represent the subsystem of displacement 
and Lagrange multiplier in the context of contact computation.
To solve the displacement subsystem, the AMG method is used, and the coarsening strategy is PMIS.
The ILU method is used to solve the Lagrange multiplier subsystem~\cite{AMGwiesner2021algebraic}.

When AMG method is used as a preconditioner for solving 
the saddle linear system, 
we use HMIS coarsening method~\cite{de2006reducing}, 
the extended+$i$ interpolation method~\cite{de2008distance}
and the node coarsening strategy~\cite{AMGbaker2016scalability}.

The solution results for the three models are presented in Table~\ref{tab:compare}, with all tests conducted using 28 processes. Models 1 and 2 have problem sizes of 639,382, while model 3 has a problem size of 712,322. In numerical experiments, the maximum number of iterations is 2000.

\begin{table}[H]
  \centering
  \setlength{\tabcolsep}{6mm}
  \caption{Comparison for different preconditioners}
    \begin{tabular}{clccc}
    \toprule
    Model & Preconditioner   & \NIT  & \TIME  & $r_{\text{rel}}^k$ \\
    \midrule
    \multirow{4}[1]{*}{model 1} & SIMPLE & 2000  & 53.39  & 9.99$\times 10^{-1}$ \\
          & AMG   & 2000  & 25.48  & 8.54$\times 10^{-6}$ \\
          & TLAMG:$\tilde{\mathcal{P}}^{\text{d}} (\mathcal{B}_{F})$ & 25    & 1.57  & 4.11$\times 10^{-9}$ \\
          & TLAMG:$\tilde{\mathcal{P}}^{\text{d}} (\mathcal{B}_{s})$ & 26    & 1.79  & 7.23$\times 10^{-9}$ \\
    \midrule
    \multirow{4}[0]{*}{model 2} & SIMPLE & 261   & 8.17  & 9.47$\times 10^{-9}$ \\
          & AMG   & 2000  & 25.74  & 6.49$\times 10^{-7}$ \\
          & TLAMG:$\tilde{\mathcal{P}}^{\text{d}} (\mathcal{B}_{F})$ & 24    & 1.63  & 8.31$\times 10^{-9}$ \\
          & TLAMG:$\tilde{\mathcal{P}}^{\text{d}} (\mathcal{B}_{s})$ & 26    & 1.85  & 5.11$\times 10^{-9}$ \\
    \midrule
    \multirow{4}[1]{*}{model 3} & SIMPLE & 2000  & 60.90  & 9.99$\times 10^{-1}$ \\
          & AMG   & 2000  & 28.64  & 6.51$\times 10^{-6}$ \\
          & TLAMG:$\tilde{\mathcal{P}}^{\text{d}} (\mathcal{B}_{F})$ & 30    & 1.79  & 5.18$\times 10^{-9}$ \\
          & TLAMG:$\tilde{\mathcal{P}}^{\text{d}} (\mathcal{B}_{s})$ & 29    & 2.00  & 5.31$\times 10^{-9}$ \\
    \bottomrule
    \end{tabular}%
  \label{tab:compare}%
\end{table}%

Because it applies only when the left-upper block is nonsingular, the SIMPLE preconditioner is exclusively suitable for model 2, achieving convergence in 261 iterations within 8.17 seconds.
The classical AMG method failed to converge within 2000 iterations. 
In contrast, both two-level preconditioning methods achieved convergence within 30 iterations and solved each model in less than 2.00 seconds. 
Therefore, the two-level methods prove to be effective preconditioners for solving saddle point systems in contact mechanics.
They are applicable across various contact models, offering advantages such as reduced iteration counts and robust efficiency.

\subsection{Algorithm scalability and parallel scalability}
\label{subsect:xxx}
In the final part of the numerical experiments, we explore the algorithmic scalability and parallel scalability of the two-level preconditioning methods. 
To analyze algorithmic scalability, we tested three different problem sizes for each contact model using 56 processes. It is observed from Table~\ref{tab:diff-dofs} that as the problem size increases, the iteration of both types of two-level methods remain  stable, while solution times increase proportionally with the increase in degrees of freedom.
For example, for model 1, the problem size increased from 10.07 million to 22.48 million, which is a 2.2-fold increase. The solve time for $\tilde{\mathcal{P}}^{\text{d}} (\mathcal{B}_{s})$ in the two-level method increased from 16.67 seconds to 32.17 seconds, approximately doubling.
Therefore, the two-level preconditioning method exhibits good algorithmic scalability.

Furthermore, the results indicate that in most cases, solve times with the sSIMPLE smoother $\mathcal{B}_{s}$ are shorter than those with the  $\mathcal{B}_{F}$ smoother.

\begin{table}[H]
  \centering
  \setlength{\tabcolsep}{5mm}
  \renewcommand{\arraystretch}{0.96} 
  \caption{Iteration numbers and solution time of two-level preconditioned GCR method for solving contact systems with different DOFs (56 Cores)}
    \begin{tabular}{ccrccc}
    \toprule
    Model & Method & \multicolumn{1}{c}{DOFs} & \NIT   & \TIME     & $r_{\text{rel}}^k$ \\
    \midrule
    \multirow{6}[1]{*}{model 1} & \multirow{3}[1]{*}{$\tilde{\mathcal{P}}^{\text{d}} (\mathcal{B}_{s})$} & 5 661 334 & 30    & 9.83  & 7.19$\times 10^{-9}$ \\
          &       & 10 078 520 & 31    & 16.67  & 6.43$\times 10^{-9}$ \\
          &       & 22 483 048 & 35    & 32.15  & 5.66$\times 10^{-9}$ \\
          \cmidrule{2-6}
          & \multirow{3}[1]{*}{$\tilde{\mathcal{P}}^{\text{d}} (\mathcal{B}_{F})$} & 5 661 334 & 28 & 9.50   &  6.92$\times 10^{-9}$ \\
          &       & 10 078 520 & 32    & 16.79  & 4.78$\times 10^{-9}$ \\
          &       & 22 483 048 & 32    & 33.86  & 4.99$\times 10^{-9}$ \\
    \midrule
    \multirow{6}[1]{*}{model 2} & \multirow{3}[1]{*}{$\tilde{\mathcal{P}}^{\text{d}} (\mathcal{B}_{s})$} & 5 661 334 & 29    & 8.50  & 4.96$\times 10^{-9}$ \\
          &       & 10 078 520 & 30    & 15.86  & 5.65$\times 10^{-9}$ \\
          &       & 22 483 048 & 33    & 33.02  & 9.58$\times 10^{-9}$ \\
          \cmidrule{2-6}
          & \multirow{3}[1]{*}{$\tilde{\mathcal{P}}^{\text{d}} (\mathcal{B}_{F})$} & 5 661 334 & 27    & 8.77  & 5.66$\times 10^{-9}$ \\
          &       & 10 078 520 & 31    & 16.25  & 6.16$\times 10^{-9}$ \\
          &       & 22 483 048 & 30    & 33.78  & 9.88$\times 10^{-9}$ \\
    \midrule
    \multirow{6}[1]{*}{model 3} & \multirow{3}[1]{*}{$\tilde{\mathcal{P}}^{\text{d}} (\mathcal{B}_{s})$} & 5 351 164 & 33    & 8.51  & 8.98$\times 10^{-9}$ \\
          &       & 11 369 428 & 35    & 20.62  & 7.80$\times 10^{-9}$ \\
          &       & 22 465 866 & 36    & 39.59  & 5.49$\times 10^{-9}$ \\
          \cmidrule{2-6}
          & \multirow{3}[1]{*}{$\tilde{\mathcal{P}}^{\text{d}} (\mathcal{B}_{F})$} & 5 351 164 & 33    & 9.76  & 8.23$\times 10^{-9}$ \\
          &       & 11 369 428 & 43    & 20.29  & 7.11$\times 10^{-9}$ \\
          &       & 22 465 866 & 41    & 40.48  & 9.09$\times 10^{-9}$ \\
    \bottomrule
    \end{tabular}%
  \label{tab:diff-dofs}%
\end{table}%

Taking model 2 as an example, the results for solving with 22.48 million DOFs using different numbers of processor cores are presented. Table~\ref{tab:diff-np} shows the iteration, relative residuals, solution times, speedup ratio (denoted as $S_p$), and parallel efficiency (denoted as $\eta$ (\%)) for both methods. 
Table~\ref{tab:diff-np} indicates that for the $\tilde{\mathcal{P}}^{\text{d}} (\mathcal{B}_{s})$ method, the total solution time decreases from 100.54s to 33.02s as the number of processor cores increases from 14 to 56. The speedup ratio increases from 1 to 3.04, while the parallel efficiency decreases from 100\% to 76.11\%.
However, with a further increase to 112 cores, the total solution time increases, and the parallel efficiency drops to 24.19\%. This suggests that the method scales effectively up to 56 cores; beyond this, the communication overhead increases significantly when using 112 cores to solve the problem.
Similar conclusions apply to the $\tilde{\mathcal{P}}^{\text{d}} (\mathcal{B}_{F})$ method.

\begin{table}[H]
  \centering
  \setlength{\tabcolsep}{5mm}
  \caption{Parallel performance of the two-level preconditioned GCR method for solving model 2  (22 483 048 DOFs)}
    \begin{tabular}{ccccccc}
    \toprule
    Method & \NP & \NIT   & \TIME     & $r_{\text{rel}}^k$   & $S_p$   & $\eta$(\%) \\
    \midrule
    \multirow{4}[1]{*}{$\tilde{\mathcal{P}}^{\text{d}} (\mathcal{B}_{s})$} & 14    & 31    & 100.54  & 9.59$\times 10^{-9}$ & 1.00  & 100 \\
          & 28    & 32    & 60.01  & 8.89$\times 10^{-9}$ & 1.68  & 83.77  \\
          & 56    & 33    & 33.02  & 9.58$\times 10^{-9}$ & 3.04  & 76.11  \\
          & 112   & 33    & 51.94  & 7.60$\times 10^{-9}$ & 1.94  & 24.19  \\
    \midrule
    \multirow{4}[0]{*}{$\tilde{\mathcal{P}}^{\text{d}} (\mathcal{B}_{F})$} & 14    & 29    & 105.40  & 8.37$\times 10^{-9}$ & 1.00  & 100 \\
          & 28    & 30    & 60.28  & 5.51$\times 10^{-9}$ & 1.75  & 87.43  \\
          & 56    & 30    & 33.78  & 9.88$\times 10^{-9}$ & 3.12  & 78.02  \\
          & 112   & 30    & 52.98  & 6.76$\times 10^{-9}$ & 1.99  & 24.87  \\
    \bottomrule
    \end{tabular}%
  \label{tab:diff-np}%
\end{table}%

\section{Conclusion and Remarks}
\label{sect:remark}

In contact computations, solving linear equations is of paramount importance. When 
employing the Lagrange multiplier method to impose contact constraints, the resulting 
saddle point system is difficult to solve because of the indefiniteness of its 
coefficient matrix. Based on the characteristics of contact computations, this study 
presents a two-level preconditioning method.

This method utilizes physical quantities for coarsening, selecting specific displacement 
degrees of freedom as coarse points. The study introduces two types of interpolation 
operators and corresponding smoothing algorithms. Additionally, the coarse grid matrix 
exhibits positive definiteness, facilitating the solution using the algebraic multigrid 
(AMG) method.

During the method's implementation, interpolation computations are optimized, taking 
into account the matrix structure of contact problems, thus ensuring ease of use. 
Moreover, by eliminating small element values from both the interpolation and coarse 
grid operators, the operator complexity of the two-level method is effectively reduced, 
improving computational efficiency. Numerical results for three contact models show that
the GCR method with two-level preconditioning significantly reduces solution time 
compared to both the AMG and SIMPLE methods. This demonstrates that two-level 
preconditioning represents an effective approach for contact saddle point problems.



\bibliographystyle{elsarticle-num} 
\bibliography{twolevel}





\end{document}